\documentclass[a4paper,onecolumn, superscriptaddress,10pt, accepted=2020-05-27,issue=4, volume=2, shorttitle=papers/compositionality-2-4]{compositionalityarticle}
\pdfoutput=1 

\usepackage[utf8]{inputenc}
\usepackage[english]{babel}
\usepackage[T1]{fontenc}
\usepackage{amsmath}
\usepackage{hyperref}

\usepackage{tikz}
\usepackage{lipsum}

\usepackage{a4wide}

\usepackage{mathrsfs}
\usepackage{etex}
\usepackage{mathtools}
\usepackage{latexsym}
\usepackage{amssymb}
\usepackage{amsthm}
\usepackage{amsmath}
\usepackage{mathabx}

\usepackage{colortbl}

\usepackage[all]{xy}
\usepackage{verbatim}
\usepackage{listings}
\usepackage{fancyvrb}

\usepackage{graphicx}

\usepackage{accents} 
\usepackage{enumerate}
\usepackage{wrapfig}
\usepackage{tikz}
\usepackage{tikz-cd}
\usetikzlibrary{automata,shapes,arrows,matrix,backgrounds,positioning,plotmarks,calc,patterns,matrix,decorations.pathreplacing,decorations.pathmorphing,decorations.text,decorations.markings}
\usepackage[colorinlistoftodos,shadow]{todonotes}

\usepackage{multirow}
\usepackage{mdwlist}

\usepackage{stmaryrd}
\usepackage{mathdots} 

\usepackage{fancyvrb}

\usepackage[toc,page]{appendix}
\usepackage{float}

\usepackage{extarrows}
\usepackage{pdflscape}
\usepackage{rotating}
\usepackage{etex}
\newtheoremstyle{mytheoremstyle} 
    {5pt}                    
    {5pt}                    
    {\itshape}                   
    {\parindent}                           
    {\bf}                   
    {.}                          
    {.5em}                       
    {}  

\theoremstyle{mytheoremstyle}

\newtheorem{theorem}{Theorem}[section]
\newtheorem{theoremdefinition}[theorem]{Theorem and Definition}

\newtheorem{lemma}[theorem]{Lemma}

\newtheorem{corollary}[theorem]{Corollary}

\newtheoremstyle{mytdefintionstyle} 
    {5pt}                    
    {5pt}                    
    {\rm}                   
    {\parindent}                           
    {\bf}                   
    {.}                          
    {.5em}                       
    {}  

\theoremstyle{remark}
\newtheorem{remark}[theorem]{Remark}

\theoremstyle{mytdefintionstyle}
\newtheorem{definition}[theorem]{Definition}

\newtheorem{example}[theorem]{Example}

\newtheorem{construction}[theorem]{Construction}
\newtheorem*{convention}{Convention}
\newtheorem{notation}[theorem]{Notation}

\newtheoremstyle{exmp_contd} 
{\topsep} {\topsep}%
{\upshape}
{}
{\bfseries}
{}
{ }
{\thmname{#1}\,\thmnumber{ #2}\thmnote{#3}\enspace(continued)}

\theoremstyle{exmp_contd}

\usepackage{xspace}
\definecolor{ExQ}{HTML}{0000FF}
\definecolor{Dec}{HTML}{E07B00}

\newcommand{\CapPkg}{\textsc{Cap}\xspace}

\newcommand{\AC}{\mathbf{A}}
\newcommand{\BC}{\mathbf{B}}

\newcommand{\FC}{\mathbf{F}}
\newcommand{\PC}{\mathbf{P}}

\newcommand{\pmatrow}[2]{ \begin{pmatrix}{#1} & {#2} \end{pmatrix} }
\newcommand{\pmatcol}[2]{ \begin{pmatrix}{#1} \\ {#2} \end{pmatrix} }

\newcommand{\CIC}{\mathcal{Q}}
\newcommand{\ACIC}{\mathbf{Aux}}

\newcommand{\N}{\mathbb{N}}
\newcommand{\Nzero}{\mathbb{Z}_{\geq 0}}

\newcommand{\Q}{\mathbb{Q}}

\newcommand{\SyzC}{\mathbf{Syz}}

\newcommand{\Ab}{\mathbf{Ab}}

\newcommand{\op}{\mathrm{op}}

\newcommand{\Freyd}{\mathcal{A}}

\newcommand{\im}{\mathrm{im}}

\newcommand{\id}{\mathrm{id}}

\DeclareMathOperator{\Hom}{\mathrm{Hom}}

\DeclareMathOperator{\kernel}{\mathrm{ker}}

\DeclareMathOperator{\cokernel}{\mathrm{coker}}

\DeclareMathOperator{\image}{\mathrm{im}}

\newcommand{\emb}{\mathrm{emb}}

\newcommand{\Rows}{\mathrm{Rows}}

\newcommand{\grRows}{\mathrm{grRows}}

\newcommand{\fp}{\mathrm{fp}}

\newcommand{\Modl}{\text{-}\mathrm{Mod}}
\newcommand{\Modr}{\mathrm{Mod}\text{-}}

\newcommand{\operation}{operation\xspace}

\newcommand{\constructs}{constructs\xspace}

\tikzset{round left paren/.style={ncbar=0.5cm,out=120,in=-120}}
\tikzset{round right paren/.style={ncbar=0.5cm,out=60,in=-60}}
\newcolumntype{C}[1]{>{\centering\arraybackslash$}p{#1}<{$}}
\newlength{\mycolwd}
\usepackage{array, xcolor}
\definecolor{lightgray}{gray}{0.8}
\newcolumntype{L}{>{\raggedleft}p{0.28\textwidth}}
\newcolumntype{R}{p{0.8\textwidth}}

\definecolor{ctcolor}{gray}{0.95}

\definecolor{ctucolor}{gray}{0.85}

\makeatletter
\newcommand{\thickhline}{%
    \noalign {\ifnum 0=`}\fi \hrule height 1pt
    \futurelet \reserved@a \@xhline
}
\newcolumntype{"}{@{\hskip\tabcolsep\vrule width 1pt\hskip\tabcolsep}}
\makeatother

\usepackage{enumitem}
\newlist{theoremenumerate}{enumerate}{1}
\setlist[theoremenumerate]{label=(\arabic{theoremenumeratei}), ref=\thetheorem.(\arabic{theoremenumeratei}),noitemsep}

\usepackage[numbers,sort&compress]{natbib}

\begin{document}

\title{Closing the category of finitely presented functors under images made constructive}
\date{}
\author{Sebastian Posur}
\email{posur@art.rwth-aachen.de}
\homepage{https://sebastianpos.github.io/}
\orcid{0000-0002-9673-0811}
\thanks{The author is supported by Deutsche Forschungsgemeinschaft (DFG) grant SFB-TRR 195: \emph{Symbolic Tools in Mathematics and their Application}.}
\affiliation{RWTH Aachen University, Germany}
\maketitle

\begin{abstract}
For an additive category $\PC$
we provide an explicit construction of a category
$\CIC( \PC )$
whose objects can be thought of as formally representing
$\frac{\image( \gamma )}{\image( \rho ) \cap \image( \gamma )}$ for given morphisms $\gamma: A \rightarrow B$
and $\rho: C \rightarrow B$ in $\PC$, even though $\PC$ does not need to admit quotients or images.
We show how it is possible to calculate effectively within $\CIC( \PC )$,
provided that a basic problem related to syzygies can be handled algorithmically.
We prove an equivalence of $\CIC( \PC )$ with the smallest subcategory
of the category of contravariant functors from $\PC$ to the category of abelian groups $\Ab$ 
which contains all finitely presented functors
and is closed under the operation of taking images. Moreover, we characterize the abelian case:
$\CIC( \PC )$ is abelian if and only if it is equivalent to $\fp( \PC^{\op}, \Ab )$,
the category of all finitely presented functors, which in turn, by a theorem of Freyd,
is abelian if and only if $\PC$ has weak kernels.

The category $\CIC( \PC )$ is a categorical abstraction of the data structure for finitely presented $R$-modules
employed by the computer algebra system \textsc{Macaulay2}, where $R$ is a ring. By our generalization
to arbitrary additive categories, we show how this data structure can also be used for modeling finitely presented graded modules,
finitely presented functors, and some not necessarily finitely presented modules over a non-coherent ring.
\end{abstract}

\tableofcontents

\section{Introduction}
The purpose of constructive category theory lies in
finding categorical representations (data structures) of mathematical objects
such that effective computations become possible \cite{PosCCT}.
A nice example of this philosophy is provided by the case of finitely presented modules over a ring $R$:
it only requires some basic algorithms for $R$
in order to obtain an effective categorical framework for doing homological algebra \cite{BL}
that even allows the implementation of concepts like
spectral sequences \cite{barhabil, PosurDoktor}, Serre quotients \cite{BL_GabrielMorphisms, GutscheDoktor}, 
or the grade filtration \cite{QGrade}.

Regarding $\Ab$-categories\footnote{These are categories enriched over the category of abelian groups $\Ab$.}
as ``rings with several objects''
is a powerful idea 
thoroughly developed by Mitchell in \cite{MitchellRings}
that yielded remarkable generalizations and clarifications in homological ring theory.
Following the idea of generalizing from a ring $R$ to an $\Ab$-category $\PC$,
the purpose of this paper is to explain, from a constructive and categorical point of view,
the data structure for modules over a ring $R$ used by the computer algebra system \textsc{Macaulay2} \cite{M2},
and moreover to generalize this data structure from the case of $R$ to the case of $\PC$.
The upshot is an effective treatment of the smallest subcategory of the category of contravariant additive functors $\PC \rightarrow \Ab$
which contains all finitely presented functors and is closed under images.

In \textsc{Macaulay2}, the data structure of a module is given by
two matrices $A \in R^{a \times b}$ and $C \in R^{c \times b}$
for $a,b,c \in \Nzero$. The left
$R$-module corresponding to such a pair of matrices is the (abstract) subquotient module $\frac{\im(A)}{\im(A) \cap \im(C)}$,
or equivalently $\frac{\im(A) + \im(C)}{\im(C)}$, of the row module $R^{1 \times b}$,
where we identify a matrix with its induced morphism between free row modules.
Given a second pair of matrices $A' \in R^{a' \times b'}$ and $C' \in R^{c' \times b'}$,
a morphism from
$\frac{\im(A)}{\im(A) \cap \im(C)}$
to
$\frac{\im(A')}{\im(A') \cap \im(C')}$
is modeled by a matrix $M \in R^{a \times a'}$
such that we may complete the following square with the dashed arrow
to a commutative diagram:
\begin{center}
            \begin{tikzpicture}[label/.style={postaction={
              decorate,
              decoration={markings, mark=at position .5 with \node #1;}}},
              mylabel/.style={thick, draw=none, align=center, minimum width=0.5cm, minimum height=0.5cm,fill=white}]
                \coordinate (r) at (3.5,0);
                \coordinate (u) at (0,2);
                \node (A) {$R^{1 \times a}$};
                \node (B) at ($(A)+1*(r)$) {};
                \node (C) at ($(B) +(r)$) {};
                \node (A2) at ($(A) - (u)$) {$R^{1 \times a'}$};
                \node (B2) at ($(B) - (u)$) {$\frac{\im(A')}{\im(A') \cap \im(C')}$};
                \node (C2) at ($(C) - (r)$) {$\frac{\im(A)}{\im(A) \cap \im(C)}$};
                \draw[->>,thick] (A2) -- (B2);
                \draw[->,thick, dashed] (C2) --  (B2);
                \draw[->>,thick] (A) --  (C2);
                \draw[->,thick] (A) --node[left]{$M$} (A2);
            \end{tikzpicture}
\end{center}
This fact can be technically expressed as follows:
for all $\sigma \in R^{1 \times a}$, $\omega \in R^{1 \times c}$ such that
$\sigma \cdot A = \omega \cdot C$,
there exists $\omega' \in  R^{1 \times c'}$
such that $(\sigma \cdot M) \cdot A' = \omega' \cdot C'$.
In other words, a syzygy $\sigma$, i.e.,
an element in the kernel of $R^{1 \times a} \twoheadrightarrow \frac{\im(A)}{\im(A) \cap \im(C)}$,
is mapped via $M$ to another syzygy $\sigma \cdot M$,
i.e., an element in the kernel of $R^{1 \times a'} \twoheadrightarrow \frac{\im(A')}{\im(A') \cap \im(C')}$.
Moreover, the technical condition for $M$ representing the zero morphism
is the following: there exists $\zeta \in R^{a \times c'}$ such that
$M \cdot A' = \zeta \cdot C'$.
In other words, every element in the image of the induced morphism is already a syzygy.
Now, the fact that these conditions can be expressed purely in the language of matrices
over $R$ is our starting point for generalizing this data structure to an arbitrary additive category.

Note that matrices over $R$ form the morphisms of an additive category $\Rows_R$,
which is the full subcategory of all $R$-modules generated by the row modules $R^{1 \times n}$, $n \in \Nzero$.
If we think of the matrices in our description of the module data structure
as morphisms in $\Rows_R$, then we can easily replace $\Rows_R$ by an arbitrary additive category $\PC$
in order to obtain a new category $\CIC( \PC )$, whose objects
are pairs of morphisms $(A \longrightarrow B \longleftarrow C)$ in $\PC$ having the same range,
a so-called cospan.
A morphism from $(A \longrightarrow B \longleftarrow C)$ to $(A' \longrightarrow B' \longleftarrow C')$
is given by a morphism $A \longrightarrow A'$ in $\PC$
that respects syzygies, a condition which can formally be expressed similarly to the corresponding condition in the case of matrices over $R$.
We interpret the objects
$(A \stackrel{\gamma}{\longrightarrow} B \stackrel{\rho}{\longleftarrow} C)$ of $\CIC( \PC )$
as entities that ``behave'' like the subquotient $\frac{\im(\gamma)}{\im(\gamma) \cap \im(\rho)}$,
even though neither images nor quotients do have to exist in $\PC$.

In this paper, whenever we describe the constructive aspects of the presented theory,
we appeal to an intuitive understanding of the concept of an algorithm or a data structure, see \cite[Introduction]{MRRConstructiveAlgebra}.
All constructions are written in a way such that an implementation in a software project like
\CapPkg (categories, algorithms, programming) \cite{CAP-project} becomes possible.

In Section \ref{section:subquotients}, we formally construct the category $\CIC( \PC )$
and describe the main algorithmic problem one needs to be able to solve within $\PC$
in order to be able to work algorithmically with $\CIC( \PC )$: the so-called syzygy inclusion problem (see Definition \ref{definition:syzygy_inclusion}).
If $\PC$ has decidable syzygy inclusion,
we show how to compute cokernels, universal epi-mono factorizations, lifts along monomorphisms,
and colifts along epimorphisms in $\CIC( \PC )$.

In Section \ref{section:modules}, we prove (Corollary \ref{corollary:characterization_as_cokernel_image_closure})
that $\CIC( \PC )$ identifies with
the smallest full and replete subcategory of the category of all additive functors
$\PC^{\mathrm{op}} \rightarrow \Ab$ (mapping to the category of abelian groups $\Ab$)
which contains the representable functors $\Hom_{\PC}(-,A)$ for $A \in \PC$
and is closed under the operations of taking cokernels and images.
In particular, we get a full and faithful functor
\[
 \fp( \PC^{\op}, \Ab ) \hookrightarrow \CIC( \PC )
\]
which realizes the category of all finitely presented functors $\fp( \PC^{\op}, \Ab )$
as a full subcategory of $\CIC( \PC )$.
If $\PC = \Rows_R$, then contravariant additive functors to $\Ab$ identify with $R$-modules,
and $\fp( \Rows_R^{\op}, \Ab )$ with the category of finitely presented $R$-modules.
In this case, $\CIC( \Rows_R )$ can be seen as the smallest full and replete subcategory of all $R$-modules
that contains the row modules $R^{1 \times n}$ for all $n \geq 0$ and is closed under cokernels and images.

By a theorem of Freyd \cite{FreydRep}, $\fp( \PC^{\op}, \Ab )$ is an abelian category if and only if 
$\PC$ has weak kernels. We prove that the same characterization holds for $\CIC( \PC )$ (Theorem \ref{theorem:abelian_case}),
and show explicitly how weak kernels can be used to construct kernels in $\CIC( \PC )$ in Section \ref{section:abelian_case}.
We also introduce the notion of a biased weak pullback in $\PC$, which, from an algorithmic point of view,
turns out to be more effective in the construction of kernels in $\CIC( \PC )$.
Finally, we prove that $\fp( \PC^{\op}, \Ab )$ and $\CIC( \PC )$
are equivalent as abstract categories if and only if $\CIC( \PC )$ is abelian,
which, as a byproduct, yields an interesting result that only concerns the category $\fp( \PC^{\op}, \Ab )$:
it is abelian if and only if it has epi-mono factorizations.

In the last Section \ref{section:applications},
we give an example of a non-coherent ring $R$, i.e., a ring such that the category $\Rows_R$
does not admit weak kernels, but which nevertheless has decidable syzygy inclusion (Theorem \ref{theorem:rows_R_decidable_syzygy_inclusion}).
It follows from our discussion in Section \ref{section:abelian_case} that the inclusion
$
 \fp( \Rows_R^{\op}, \Ab ) \hookrightarrow \CIC( \Rows_R )
$
is proper.
Thus, we may algorithmically perform all the constructions listed in Section \ref{section:subquotients} within $\CIC( \Rows_R )$,
and this for a greater class of $R$-modules than finitely presented ones.

To conclude, we discuss how our category constructor $\CIC( - )$ can also yield a computational model
for graded modules, and for finitely presented functor categories on module categories
by an iterated application.

\begin{convention}
 Given morphisms $\gamma_{AC}: A \rightarrow C$, $\gamma_{AD}: A \rightarrow D$,
$\gamma_{BC}: B \rightarrow C$, and $\gamma_{BD}: B \rightarrow D$
in an additive category $\PC$,
we denote the induced morphism between direct sums using the row convention, i.e.,
\[
 \begin{pmatrix}
  \gamma_{AC} & \gamma_{AD} \\
  \gamma_{BC} & \gamma_{BD}
 \end{pmatrix}:
 A \oplus B \rightarrow C \oplus D.
\]
We use the notation
$\alpha \cdot \beta: A \rightarrow C$
for the composition of morphisms $\alpha: A \rightarrow B$ and $\beta: B \rightarrow C$,
since then, composition of morphisms between direct sums simply becomes matrix multiplication.

Given two subobjects $U,V \hookrightarrow W$ in an abelian category,
we use the simplified notation $\frac{U}{V}$ in order to denote the subquotient $\frac{U+V}{V} \simeq \frac{U}{U \cap V}$ of $W$.
We also occasionally use the standard abbreviations epis, monos, isos
for epimorphisms, monomorphisms, and isomorphisms, respectively.
A mono that arises as the kernel of some morphism in a pointed category
is called a normal mono.

A universal epi-mono factorization is an essentially unique factorization of a morphism
into an epi followed by a mono. For brevity we also refer to such a factorization as an epi-mono factorization.

Throughout the paper, a functor between two additive categories is always meant to be an additive functor.

The symbol $\Nzero$ denotes the set of non-negative integers.
\end{convention}

\section{The category $\CIC(\PC)$}\label{section:subquotients}

In this section, $\PC$ always denotes an additive category.
The goal is to formally construct an additive category $\CIC(\PC)$
that admits cokernels and epi-mono factorizations
together with a full additive embedding $\PC \subseteq \CIC(\PC)$.
As a running example, the reader can think of $\PC$
as $\Rows_R$, i.e., the full subcategory
of $R$-modules $R\Modl$ generated by row modules $R^{1 \times n}$ for $n \in \Nzero$,
where $R$ is any unital ring. Morphisms in $\Rows_R$ will be tacitly identified with matrices over $R$.
The category $\CIC( \Rows_R )$ will turn out
to be equivalent to the smallest full and replete subcategory of $R\Modl$ that contains $\Rows_R$
and is closed under taking cokernels and images in $R\Modl$.

\subsection{The category of syzygies}
A \textbf{cospan} in $\PC$ is simply a pair of morphisms
\[(A \stackrel{\gamma}{\longrightarrow} B, C \stackrel{\rho}{\longrightarrow} B)\] in $\PC$,
with shorthand notation $(A \stackrel{\gamma}{\longrightarrow} B \stackrel{\rho}{\longleftarrow} C)$.

\begin{definition}
 Let $(A \stackrel{\gamma}{\longrightarrow} B \stackrel{\rho}{\longleftarrow} C)$ be a cospan in $\PC$.
 Its \textbf{category of syzygies} $\SyzC( A \stackrel{\gamma}{\longrightarrow} B \stackrel{\rho}{\longleftarrow} C )$
 consists of the following data:
 \begin{enumerate}
  \item Objects, which we also call \textbf{syzygies}, are given by morphisms $S \stackrel{\sigma}{\longrightarrow} A$ in $\PC$ 
  such that there exists another morphism
        $\omega: S \longrightarrow C$, which we call a \textbf{syzygy witness},
        rendering the diagram
        \begin{center}
            \begin{tikzpicture}[label/.style={postaction={
              decorate,
              decoration={markings, mark=at position .5 with \node #1;}}},
              mylabel/.style={thick, draw=none, align=center, minimum width=0.5cm, minimum height=0.5cm,fill=white}]
                \coordinate (r) at (2.5,0);
                \coordinate (u) at (0,1.2);
                \node (A) {$S$};
                \node (B) at ($(A)+1*(r)$) {};
                \node (C) at ($(B) +(r)$) {};
                \node (A2) at ($(A) - (u)$) {$A$};
                \node (B2) at ($(B) - (u)$) {$B$};
                \node (C2) at ($(C) - (u)$) {$C$};
                \draw[->,thick] (A2) -- node[below]{$\gamma$} (B2);
                \draw[->,thick] (C2) -- node[below]{$\rho$} (B2);
                \draw[->,thick, dashed] (A) --node[above]{$\omega$}  (C2);
                \draw[->,thick] (A) --node[left]{$\sigma$} (A2);
            \end{tikzpicture}
        \end{center}
        commutative.
        Whenever we depict a syzygy by a commutative diagram like the one above, we will
        draw the syzygy witness with a dashed arrow.
  \item A morphism from a syzygy $S \stackrel{\sigma}{\longrightarrow} A$ to a syzygy $S' \stackrel{\sigma'}{\longrightarrow} A$
        is given by a morphism $\tau: S \rightarrow S'$ such that $\tau \cdot \sigma' = \sigma$, i.e., the following diagram commutes:
        \begin{center}
            \begin{tikzpicture}[label/.style={postaction={
              decorate,
              decoration={markings, mark=at position .5 with \node #1;}}},
              mylabel/.style={thick, draw=none, align=center, minimum width=0.5cm, minimum height=0.5cm,fill=white}]
                \coordinate (r) at (2.5,0);
                \coordinate (u) at (0,1.2);
                \node (A) {$S$};
                \node (B) at ($(A)+2*(r)$) {$S'$};
                \node (C) at ($(A) +(r) - (u)$) {$A$.};
                
                \draw[->,thick] (A) -- node[above]{$\tau$} (B);
                \draw[->,thick] (A) -- node[below]{$\sigma$} (C);
                \draw[->,thick] (B) -- node[below]{$\sigma'$} (C);
            \end{tikzpicture}
        \end{center}
 \end{enumerate}
\end{definition}

\begin{remark}
 The category of syzygies $\SyzC( A \stackrel{\gamma}{\longrightarrow} B \stackrel{\rho}{\longleftarrow} C )$
 is a full subcategory of the slice category of $\PC$ over the object $A$.
\end{remark}

\begin{example}
 In our running example $\PC = \Rows_R$,
 giving a cospan means giving a pair of matrices
 $(R^{1 \times a} \stackrel{\gamma}{\longrightarrow} R^{1 \times b} \stackrel{\rho}{\longleftarrow} R^{1 \times c})$
 that have the same number of columns.
 An object in its category of syzygies is a matrix $R^{1 \times s} \stackrel{\sigma}{\longrightarrow} R^{1 \times a}$
 which fits into a chain complex
 \begin{center}
    \begin{tikzpicture}[label/.style={postaction={
      decorate,
      decoration={markings, mark=at position .5 with \node #1;}}},
      mylabel/.style={thick, draw=none, align=center, minimum width=0.5cm, minimum height=0.5cm,fill=white}]
        \coordinate (r) at (2.5,0);
        \coordinate (u) at (0,1.2);
        \node (A) {$R^{1 \times s}$};
        \node (B) at ($(A)+(r)$) {$R^{1 \times a}$};
        \node (C) at ($(B) + (r)$) {$\frac{\im( \gamma )}{\im( \rho )}$};
        \node (D) at ($(C) +(r)$) {$0$};
        \draw[->,thick] (A) --node[above]{$\sigma$} (B);
        \draw[->,thick] (B) --node[above]{$_{|}\overline{\gamma}$} (C);
        \draw[->,thick] (C) -- (D);
    \end{tikzpicture}
 \end{center}
 in $R\Modl$, hence the name category of \emph{syzygies}.
 Here, the morphism $_{|}\overline{\gamma}$ is given 
 as follows: first, we coastrict $\gamma$ to its image and obtain the morphism
 $_{|}{\gamma}: R^{1 \times a} \longrightarrow \im( \gamma )$.
 Second, we compose $_{|}{\gamma}$ with the natural projection 
 $\image( \gamma ) \twoheadrightarrow \frac{\im( \gamma )}{\im( \rho )}$
 and obtain the desired morphism $_{|}\overline{\gamma}$. Recall that
 by our convention, $\frac{\im( \gamma )}{\im( \rho )}$ is shorthand for 
 $\frac{\im( \gamma )}{\im( \rho ) \cap \im( \gamma )}$.
\end{example}

\subsection{The syzygy inclusion problem}

In this subsection, we state an algorithmic problem for $\PC$
that will turn out to be the key to a computational approach to the yet to be constructed
category $\CIC( \PC )$.

\begin{definition}\label{definition:syzygy_inclusion}
 We say that $\PC$ has \textbf{decidable syzygy inclusion}
 if it comes equipped with an algorithm
 whose input is a pair of cospans in $\PC$ with the same first object
 \begin{center}
    \begin{tikzpicture}[label/.style={postaction={
      decorate,
      decoration={markings, mark=at position .5 with \node #1;}}},
      mylabel/.style={thick, draw=none, align=center, minimum width=0.5cm, minimum height=0.5cm,fill=white}]
        \coordinate (r) at (2.5,0);
        \coordinate (u) at (0,-0.5);
        \node (A) {$A$};
        \node (B) at ($(A)+(r) - (u)$) {$B$};
        \node (C) at ($(B) + (r)$) {$C$};
        \node (Bp) at ($(A) +(r) + (u)$) {$B'$};
        \node (Cp) at ($(Bp) +(r) $) {$C'$,};
        \draw[->,thick] (A) --node[above]{$\gamma$} (B);
        \draw[->,thick] (C) --node[above]{$\rho$} (B);
        \draw[->,thick] (A) --node[below]{$\gamma'$} (Bp);
        \draw[->,thick] (Cp) --node[below]{$\rho'$} (Bp);
    \end{tikzpicture}
 \end{center}
 and whose output is a constructive answer to the question
 whether we have an inclusion
 \[
    \SyzC( A \stackrel{\gamma}{\longrightarrow} B \stackrel{\rho}{\longleftarrow} C ) \stackrel{?}{\subseteq} \SyzC( A \stackrel{\gamma'}{\longrightarrow} B' \stackrel{\rho'}{\longleftarrow} C' )
 \]
 of full subcategories of the slice category of $\PC$ over the object $A$.
 By a constructive answer, we mean that in the case when the algorithm answers affirmatively,
 it also provides
 an additional algorithm
 \[
    (S \stackrel{\sigma}{\longrightarrow} A, S \stackrel{\omega}{\longrightarrow} C) \mapsto \omega'
 \]
 mapping a syzygy $\sigma \in \SyzC( A \stackrel{\gamma}{\longrightarrow} B \stackrel{\rho}{\longleftarrow} C )$
 together with a corresponding syzygy witness $\omega$
 to
 a syzygy witness $\omega'$ that proves $\sigma \in \SyzC( A \stackrel{\gamma'}{\longrightarrow} B' \stackrel{\rho'}{\longleftarrow} C' )$.
\end{definition}

\begin{definition}
 We say that $\PC$ has \textbf{decidable lifts}
 if it comes equipped with an algorithm
 whose input is a diagram
 \begin{center}
            \begin{tikzpicture}[label/.style={postaction={
              decorate,
              decoration={markings, mark=at position .5 with \node #1;}}},
              mylabel/.style={thick, draw=none, align=center, minimum width=0.5cm, minimum height=0.5cm,fill=white}]
                \coordinate (r) at (2.5,0);
                \coordinate (u) at (0,1.2);
                \node (A) {$A$};
                \node (B) at ($(A)+1*(r)$) {};
                \node (C) at ($(B) +(r)$) {};
                \node (A2) at ($(A) - (u)$) {$B$};
                \node (B2) at ($(B) - (u)$) {$C$};
                \draw[->,thick] (B2) -- node[below]{$\gamma$} (A2);
                \draw[->,thick] (A) --node[left]{$\alpha$} (A2);
            \end{tikzpicture}
        \end{center}
 in $\PC$, and the output is either a morphism $\lambda$ rendering the diagram
 \begin{center}
            \begin{tikzpicture}[label/.style={postaction={
              decorate,
              decoration={markings, mark=at position .5 with \node #1;}}},
              mylabel/.style={thick, draw=none, align=center, minimum width=0.5cm, minimum height=0.5cm,fill=white}]
                \coordinate (r) at (2.5,0);
                \coordinate (u) at (0,1.2);
                \node (A) {$A$};
                \node (B) at ($(A)+1*(r)$) {};
                \node (C) at ($(B) +(r)$) {};
                \node (A2) at ($(A) - (u)$) {$B$};
                \node (B2) at ($(B) - (u)$) {$C$};
                \draw[->,thick] (B2) -- node[below]{$\gamma$} (A2);
                \draw[->,thick] (A) --node[left]{$\alpha$} (A2);
                \draw[->,dashed,thick] (A) --node[above]{$\lambda$} (B2);
            \end{tikzpicture}
        \end{center}
  commutative, or $\mathtt{false}$ if no such $\lambda$ exists.
\end{definition}

\begin{remark}\label{remark:decidable_lifts}
  We claim that having decidable syzygy inclusion implies having decidable lifts:
  suppose given  $(A \stackrel{\alpha}{\longrightarrow} B \stackrel{\gamma}{\longleftarrow} C)$.
  If a lift $\lambda$ exists, then any $\sigma: S \rightarrow A$ lies in $\SyzC( A \stackrel{\alpha}{\longrightarrow} B \stackrel{\gamma}{\longleftarrow} C)$
  with syzygy witness given by $\sigma \cdot \lambda$.
  Thus, we have
  \[
   \SyzC( A \longrightarrow 0 \longleftarrow 0)
   \subseteq
   \SyzC( A \stackrel{\alpha}{\longrightarrow} B \stackrel{\gamma}{\longleftarrow} C).
  \]
  Conversely, if this inclusion holds, then a lift $\lambda$ can be constructed explicitly as a syzygy witness
  of the syzygy $\id_A \in \SyzC( A \stackrel{\alpha}{\longrightarrow} B \stackrel{\gamma}{\longleftarrow} C)$.
\end{remark}

\begin{remark}
 Having decidable syzygy inclusion can also be rephrased as follows:
 $\PC$ has decidable lifts,
 and we have an algorithm that decides
 \[
    \SyzC( A \stackrel{\gamma}{\longrightarrow} B \stackrel{\rho}{\longleftarrow} C ) \stackrel{?}{\subseteq} \SyzC( A \stackrel{\gamma'}{\longrightarrow} B' \stackrel{\rho'}{\longleftarrow} C' )
 \]
 with a simple yes/no answer.
 For if the algorithm answers yes, we may produce our desired syzygy witnesses using the algorithm for computing lifts.
\end{remark}

\begin{example}\label{example:syzygy_inclusion_for_rings}
  In our running example $\PC = \Rows_R$,
  given two cospans with the same first object
  $(R^{1 \times a} \stackrel{\gamma}{\longrightarrow} R^{1 \times b} \stackrel{\rho}{\longleftarrow} R^{1 \times c})$
  and 
  $(R^{1 \times a} \stackrel{\gamma'}{\longrightarrow} R^{1 \times b'} \stackrel{\rho'}{\longleftarrow} R^{1 \times c'})$,
  being able to solve their syzygy inclusion problem implies
  being able to decide the existence of dashed arrows rendering the following diagram with exact rows commutative:
  \begin{center}
             \begin{tikzpicture}[label/.style={postaction={
               decorate,
               decoration={markings, mark=at position .5 with \node #1;}}},
               mylabel/.style={thick, draw=none, align=center, minimum width=0.5cm, minimum height=0.5cm,fill=white}]
                 \coordinate (r) at (3,0);
                 \coordinate (u) at (0,2);
                 \node (A) {$0$};
                 \node (B) at ($(A)+(r)$) {$\kernel( _{|}\overline{\gamma} )$};
                 \node (C) at ($(B)+(r)$) {$R^{1 \times a}$};
                 \node (D) at ($(C)+(r)$) {$\frac{\im( \gamma )}{\im( \rho )}$};
                 \node (E) at ($(D)+(r)$) {$0$};
                 
                 \node (A2) at ($(A)-(u)$) {$0$};
                 \node (B2) at ($(A2)+(r)$) {$\kernel( _{|}\overline{\gamma'} )$};
                 \node (C2) at ($(B2)+(r)$) {$R^{1 \times a}$};
                 \node (D2) at ($(C2)+(r)$) {$\frac{\im( \gamma' )}{\im( \rho' )}$};
                 \node (E2) at ($(D2)+(r)$) {$0$.};
                 
                 \draw[->,thick] (A) -- (B);
                 \draw[->,thick] (B) -- (C);
                 \draw[->,thick] (C) --node[above]{$_{|}\overline{\gamma}$} (D);
                 \draw[->,thick] (D) -- (E);
                 
                 \draw[->,thick] (A2) -- (B2);
                 \draw[->,thick] (B2) -- (C2);
                 \draw[->,thick] (C2) --node[above]{$_{|}\overline{\gamma'}$} (D2);
                 \draw[->,thick] (D2) -- (E2);
                 
                 \draw[->,thick,dashed] (B) -- (B2);
                 \draw[->,thick] (C) --node[left]{$\id$} (C2);
                 \draw[->,thick,dashed] (D) -- (D2);
             \end{tikzpicture}
   \end{center}
   Indeed, the rows of a syzygy $\sigma \in R^{s \times a}$ in 
   $\SyzC( R^{1 \times a} \stackrel{\gamma}{\longrightarrow} R^{1 \times b} \stackrel{\rho}{\longleftarrow} R^{1 \times c} )$
   for $s \in \Nzero$ can be regarded as a collection of $s$-many elements
   in $\kernel( _{|}\overline{\gamma} )$, and asking for the existence of the dashed arrows
   is the question of whether these rows are also lying in $\kernel( _{|}\overline{\gamma'} )$,
   which is equivalent to $\sigma$ being a syzygy in 
   $\SyzC(R^{1 \times a} \stackrel{\gamma'}{\longrightarrow} R^{1 \times b'} \stackrel{\rho'}{\longleftarrow} R^{1 \times c'})$.
   
   The question whether 
   \[\kernel( _{|}\overline{\gamma} ) \subseteq \kernel( _{|}\overline{\gamma'} )\]
   can always be answered in the case when $R$ is a (left) \textbf{computable ring},
   a notion introduced by Barakat and Lange-Hegermann in \cite{BL}.
   It is defined as a ring that comes equipped with two algorithms:
   \begin{enumerate}
    \item (Algorithm for deciding lifts): given matrices $A \in R^{m \times n}$ and $B \in R^{q \times n}$ for $m,n,q \in \Nzero$,
    decide whether there exists an $X \in R^{q \times m}$ such that
    \[
     X \cdot A = B,
    \]
    and in the affirmative case compute such an $X$.
    \item (Algorithm for computing row syzygies):
    given a matrix $A \in R^{m \times n}$, compute $o \in \Nzero$ and $L \in R^{o \times m}$ such that
    \[
     L \cdot A = 0,
    \]
    and $L$ is (weakly) universal with this property, i.e., for any other $T \in R^{p \times m}$, $p \in \Nzero$,
    such that $T \cdot A = 0$, we can find a (not necessarily unique) $U \in R^{p \times o}$ such that $U \cdot L = T$.
   \end{enumerate}
   
   Prominent examples of (commutative) computable rings are quotients of polynomial rings $k[x_1, \dots x_n]$
   for $n \in \Nzero$
   by ideals generated by finitely many prescribed polynomials, where $k$ is a computable field $k$ (like $\Q$).
   This is mainly due to Gröbner basis techniques (see, e.g., \cite{GP}).
   Also, localizations of computable commutative rings $R$ at a multiplicative set $S$ turn out to be
   computable provided that one may algorithmically
   determine witnesses for the intersection of finitely generated ideals $I \subseteq R$ with $S$ being non-empty \cite{PosLinSys}.
   
   Left computable rings are in particular left coherent, i.e., the 
   category of finitely presented left $R$-modules is abelian.
   In particular, $\kernel( _{|}\overline{\gamma} )$ and $\kernel( _{|}\overline{\gamma'} )$
   both are finitely presented modules in this case, and the computability of $R$ ensures that we 
   can algorithmically test the inclusion of the finitely many generators\footnote{
     Concretely, we may compute finitely many generators of $\kernel( _{|}\overline{\gamma} )$
     by building up the stacked matrix
     $\pmatcol{\gamma}{\rho} \in R^{(a+c) \times b}$
     and applying to it the algorithm for computing row syzygies.
     This yields a finite subset of $R^{1 \times (a+c)}$, and projecting this set
     to its first $a$ entries (via the natural projection $R^{1 \times (a+c)} \rightarrow R^{1 \times a}$) yields generators of $\kernel( _{|}\overline{\gamma} )$.
     More abstractly, $R$ being computable implies the computability of the abelian category of finitely presented $R$-modules,
     and within such a category, we can perform constructions coming from the axioms of an abelian category effectively.
     In particular, we may build up the exact sequences as they are presented within this example on the computer
     and decide the existence of the dashed arrows (see \cite{PosFreyd} for a detailed explanation).
   }
   of $\kernel( _{|}\overline{\gamma} )$
   in $\kernel( _{|}\overline{\gamma'} )$.
   
   In Section \ref{section:applications}, we will give an example of a non-coherent ring $R$
   for which $\Rows_R$ nevertheless has decidable syzygy inclusion, even though $\kernel( _{|}\overline{\gamma} )$
   might not be finitely generated.
 \end{example} 

\subsection{An auxiliary category}
We define an auxiliary additive category $\ACIC( \PC )$.
Later, $\CIC( \PC )$ will arise as a quotient of $\ACIC( \PC )$.

\begin{definition}
 The additive category $\ACIC( \PC )$ is defined by the following data:
 \begin{theoremenumerate}
  \item An object in $\ACIC( \PC )$ is given by a cospan in $\PC$.
        We will write such an object as
        \[(A \stackrel{\gamma_A}{\longrightarrow} \Omega_A \stackrel{\rho_A}{\longleftarrow} R_A),\]
        even though $\Omega_A$ and $R_A$ do not formally depend\footnote{Guided by our running example $\Rows_R$, we think of $\Omega_A$ as an ambient space for the image of $\gamma_A$, and of 
        $\rho_A: R_A \rightarrow \Omega_A$ as relations
        imposed on this ambient space.} on $A$.
  \item A morphism in $\ACIC( \PC )$ from
        $(A \stackrel{\gamma_A}{\longrightarrow} \Omega_A \stackrel{\rho_A}{\longleftarrow} R_A)$
        to
        $(B \stackrel{\gamma_B}{\longrightarrow} \Omega_B \stackrel{\rho_B}{\longleftarrow} R_B)$
        is given by a morphism
        $\alpha: A \rightarrow B$ in $\PC$
        that respects syzygies, i.e.,
        \[
         \sigma \in \SyzC( A \stackrel{\gamma_A}{\longrightarrow} \Omega_A \stackrel{\rho_A}{\longleftarrow} R_A)
         \text{\hspace{1em}implies\hspace{1em}}
         \sigma \cdot \alpha \in \SyzC( B \stackrel{\gamma_B}{\longrightarrow} \Omega_B \stackrel{\rho_B}{\longleftarrow} R_B).
        \]
        We call this the \textbf{well-definedness property} of the given morphism $\alpha$ in $\PC$
        w.r.t.\ the source $(A \stackrel{\gamma_A}{\longrightarrow} \Omega_A \stackrel{\rho_A}{\longleftarrow} R_A)$
        and range $(B \stackrel{\gamma_B}{\longrightarrow} \Omega_B \stackrel{\rho_B}{\longleftarrow} R_B)$.
        \label{definition:aux_cat_2}
        
    \item Composition and identities are inherited from $\PC$.
 \end{theoremenumerate}
\end{definition}

\begin{remark}
 If $\PC$ has decidable syzygy inclusion, then we can decide the well-definedness property by testing
        \[
         \SyzC( A \stackrel{\gamma_A}{\longrightarrow} \Omega_A \stackrel{\rho_A}{\longleftarrow} R_A)
         \subseteq
         \SyzC( A \stackrel{\alpha \cdot \gamma_B}{\longrightarrow} \Omega_B \stackrel{\rho_B}{\longleftarrow} R_B).
        \]
\end{remark}

\begin{remark}
 Composition in $\ACIC( \PC )$ is well-defined, i.e.,
 the composition of two morphisms satisfying the well-definedness property
 again satisfies the well-definedness property.
 Indeed, given two well-defined morphisms
 \[
  (A \stackrel{\gamma_A}{\longrightarrow} \Omega_A \stackrel{\rho_A}{\longleftarrow} R_A)
  \stackrel{\alpha}{\longrightarrow}
  (B \stackrel{\gamma_B}{\longrightarrow} \Omega_B \stackrel{\rho_B}{\longleftarrow} R_B)
 \]
 and
 \[
  (B \stackrel{\gamma_B}{\longrightarrow} \Omega_B \stackrel{\rho_B}{\longleftarrow} R_B)
  \stackrel{\beta}{\longrightarrow}
  (C \stackrel{\gamma_C}{\longrightarrow} \Omega_C \stackrel{\rho_C}{\longleftarrow} R_C),
 \]
 then any syzygy 
 \[\sigma \in \SyzC(A \stackrel{\gamma_A}{\longrightarrow} \Omega_A \stackrel{\rho_A}{\longleftarrow} R_A)\]
 defines a syzygy
 \[\sigma \cdot \alpha \in \SyzC(B \stackrel{\gamma_B}{\longrightarrow} \Omega_B \stackrel{\rho_B}{\longleftarrow} R_B),\]
 which in turn defines a syzygy
 \[\sigma \cdot \alpha \cdot \beta \in \SyzC(C \stackrel{\gamma_C}{\longrightarrow} \Omega_C \stackrel{\rho_C}{\longleftarrow} R_C).\]
\end{remark}

\begin{remark}
 Addition of morphisms in $\ACIC( \PC )$ is well-defined.
 Two well-defined morphisms
 \[
  (A \stackrel{\gamma_A}{\longrightarrow} \Omega_A \stackrel{\rho_A}{\longleftarrow} R_A)
  \stackrel{\alpha}{\longrightarrow}
  (B \stackrel{\gamma_B}{\longrightarrow} \Omega_B \stackrel{\rho_B}{\longleftarrow} R_B)
 \]
 and
 \[
  (A \stackrel{\gamma_A}{\longrightarrow} \Omega_A \stackrel{\rho_A}{\longleftarrow} R_A)
  \stackrel{\alpha'}{\longrightarrow}
  (B \stackrel{\gamma_B}{\longrightarrow} \Omega_B \stackrel{\rho_B}{\longleftarrow} R_B)
 \]
 yield a well-defined morphism
 \[
  (A \stackrel{\gamma_A}{\longrightarrow} \Omega_A \stackrel{\rho_A}{\longleftarrow} R_A)
  \stackrel{\alpha + \alpha'}{\longrightarrow}
  (B \stackrel{\gamma_B}{\longrightarrow} \Omega_B \stackrel{\rho_B}{\longleftarrow} R_B),
 \]
 since the sum of two syzygies having the same source is again a syzygy (simply by adding their syzygy witnesses).
 The same holds for subtraction.
 Moreover, the zero morphism
 \[
  (A \stackrel{\gamma_A}{\longrightarrow} \Omega_A \stackrel{\rho_A}{\longleftarrow} R_A)
  \stackrel{0}{\longrightarrow}
  (B \stackrel{\gamma_B}{\longrightarrow} \Omega_B \stackrel{\rho_B}{\longleftarrow} R_B)
 \]
 is always well-defined, since any morphism $S \stackrel{0}{\rightarrow} B$ is a syzygy
 in $\SyzC(B \stackrel{\gamma_B}{\longrightarrow} \Omega_B \stackrel{\rho_B}{\longleftarrow} R_B)$.
 It follows that $\ACIC( \PC )$ is an $\Ab$-category, i.e., enriched over abelian groups.
\end{remark}

\begin{remark}
 Let
 $(A_i \stackrel{\gamma_{A_i}}{\longrightarrow} \Omega_{A_i} \stackrel{\rho_{A_i}}{\longleftarrow} R_{A_i})_{i \in I}$
 be a family of objects in $\PC$
 indexed by a finite  set $I$.
 The injections $A_j \stackrel{\iota_j}{\rightarrow} \bigoplus_i A_i$ in $\PC$
 induce well-defined morphisms in $\ACIC( \PC )$
 \[
  (A_j \stackrel{\gamma_{A_j}}{\longrightarrow} \Omega_{A_j} \stackrel{\rho_{A_j}}{\longleftarrow} R_{A_j})
  \stackrel{\iota_j}{\longrightarrow }
  (\bigoplus_i A_i \stackrel{\bigoplus_i \gamma_{A_i}}{\longrightarrow} \bigoplus_i \Omega_{A_i} \stackrel{\bigoplus_i \rho_{A_i}}{\longleftarrow} \bigoplus_i R_{A_i}),
 \]
 since any syzygy $\sigma$ of the source 
 with witness $\omega$ 
 defines a syzygy $\sigma \cdot \iota_j$ of the range with witness $\omega \cdot \iota_j$.
 
 Similarly,
 the projections
  $\bigoplus_i A_i \stackrel{\pi_j}{\rightarrow} A_j$ in $\PC$ for $j \in I$
 induce well-defined morphisms
 \[
  (\bigoplus_i A_i \stackrel{\bigoplus_i \gamma_{A_i}}{\longrightarrow} \bigoplus_i \Omega_{A_i} \stackrel{\bigoplus_i \rho_{A_i}}{\longleftarrow} \bigoplus_i R_{A_i})
  \stackrel{\pi_j}{\longrightarrow }
  (A_j \stackrel{\gamma_{A_j}}{\longrightarrow} \Omega_{A_j} \stackrel{\rho_{A_j}}{\longleftarrow} R_{A_j})
 \]
 since any syzygy $\sigma$ of the source 
 with witness $\omega$ 
 defines a syzygy $\sigma \cdot \pi_j$ of the range with witness $\omega \cdot \pi_j$.
 Thus, $\ACIC( \PC )$ is also an additive category.
\end{remark}

\begin{theoremdefinition}\label{theorem:ideal_in_aux}
 Let $\mathbf{I}(\PC)$ denote the collection of all morphisms 
 \[\alpha: (A \stackrel{\gamma_A}{\longrightarrow} \Omega_A \stackrel{\rho_A}{\longleftarrow} R_A)
  \rightarrow
  (B \stackrel{\gamma_B}{\longrightarrow} \Omega_B \stackrel{\rho_B}{\longleftarrow} R_B)
 \]
 in $\ACIC( \PC )$
 such that $\alpha: A \rightarrow B$ is a syzygy in $\SyzC(B \stackrel{\gamma_B}{\longrightarrow} \Omega_B \stackrel{\rho_B}{\longleftarrow} R_B)$,
 i.e.,
 with the property that there exists a lift $\zeta$ (which we call witness for being \textbf{z}ero) rendering the diagram
  \begin{center}
    \begin{tikzpicture}[label/.style={postaction={
      decorate,
      decoration={markings, mark=at position .5 with \node #1;}}}, baseline = (D),
      mylabel/.style={thick, draw=none, align=center, minimum width=0.5cm, minimum height=0.5cm,fill=white}]
        \coordinate (r) at (2.5,0);
        \coordinate (u) at (0,1.2);
        \node (A) {$A$};
        \node (B) at ($(A)+1*(r)$) {};
        \node (C) at ($(B) +(r)$) {};
        \node (A2) at ($(A) - (u)$) {$B$};
        \node (B2) at ($(B) - (u)$) {$\Omega_B$};
        \node (C2) at ($(C) - (u)$) {$R_B$};
        \draw[->,thick] (A2) -- node[below]{$\gamma_B$} (B2);
        \draw[->,thick] (C2) -- node[below]{$\rho_B$} (B2);
        \draw[->,thick,dashed] (A) -- node[above]{$\zeta$} (C2);
        \draw[->,thick] (A) --node[left]{$\alpha$} (A2);
    \end{tikzpicture}
  \end{center}
  commutative.
  Then $\mathbf{I}(\PC)$ forms an ideal of $\ACIC( \PC )$.
\end{theoremdefinition}
\begin{proof}
 Clearly, all zero morphisms lie in $\mathbf{I}(\PC)$.
 Moreover, given addable morphisms $\alpha, \beta \in \mathbf{I}(\PC)$,
 we can add their witnesses for being zero to deduce $\alpha + \beta \in \mathbf{I}(\PC)$.
 
 Next, let
 \[
  (A \stackrel{\gamma_A}{\longrightarrow} \Omega_A \stackrel{\rho_A}{\longleftarrow} R_A)
  \stackrel{\alpha}{\longrightarrow}
  (B \stackrel{\gamma_B}{\longrightarrow} \Omega_B \stackrel{\rho_B}{\longleftarrow} R_B)
 \]
 and
 \[
  (B \stackrel{\gamma_B}{\longrightarrow} \Omega_B \stackrel{\rho_B}{\longleftarrow} R_B)
  \stackrel{\beta}{\longrightarrow}
  (C \stackrel{\gamma_C}{\longrightarrow} \Omega_C \stackrel{\rho_C}{\longleftarrow} R_C)
 \]
 be two composable morphisms in $\ACIC( \PC )$.
 If $\beta \in \mathbf{I}(\PC)$ with $\zeta$ a witness for being zero,
 then $\alpha \cdot \beta \in \mathbf{I}(\PC)$
 with $\alpha \cdot \zeta$ a witness for being zero.
 
 If $\alpha \in \mathbf{I}(\PC)$, then
 $\alpha$ is a syzygy in $\SyzC(B \stackrel{\gamma_B}{\longrightarrow} \Omega_B \stackrel{\rho_B}{\longleftarrow} R_B)$.
 By the well-definedness property of $\beta$,
 $\alpha \cdot \beta$ is a syzygy of $\SyzC(C \stackrel{\gamma_C}{\longrightarrow} \Omega_C \stackrel{\rho_C}{\longleftarrow} R_C)$,
 which also implies $\alpha \cdot \beta \in \mathbf{I}(\PC)$.
 Thus, $\mathbf{I}(\PC)$ is a collection of abelian subgroups
 closed under left and right multiplication, or in other words, an ideal of $\ACIC( \PC )$.
\end{proof}

\subsection{Definition of the category $\CIC(\PC)$}

Recall that for any additive category $\AC$ and any ideal $\mathbf{I}$ of $\AC$,
the additive quotient category $\AC/ \mathbf{I}$ has the same objects as $\AC$,
and 
\[
 \Hom_{\AC/ \mathbf{I}}( A, B ) := \Hom_{\AC}( A, B )/ \{ A \stackrel{\alpha}{\rightarrow} B \mid \alpha \in \mathbf{I} \}
\]
for all $A, B \in \AC/ \mathbf{I}$.

\begin{definition}
 We set
 \[
  \CIC(\PC) := \ACIC( \PC )/ \mathbf{I}(\PC),
 \]
 i.e., we form the additive quotient category of $\ACIC( \PC )$ by the ideal $\mathbf{I}(\PC)$.
\end{definition}

\begin{remark}
Morally, we shall think of an object 
$(A \stackrel{\gamma_A}{\longrightarrow} \Omega_A \stackrel{\rho_A}{\longleftarrow} R_A)$
in $\CIC( \PC )$ as a representation
of the quotient object ``$\frac{\im( \gamma_A )}{\im( \rho_A )}$''.
The ``Q'' in $\CIC(\PC)$ stands for \emph{quotient}.
\end{remark}

\begin{remark}
 If $\PC$ has decidable syzygy inclusion,
 then we can decide equality of morphisms in $\CIC(\PC)$.
 Deciding equality of two morphisms $\alpha$ and $\beta$ in $\CIC( \PC )$ means
 deciding whether $\alpha - \beta$ is zero, which
 is a lifting problem, which we can solve using Remark \ref{remark:decidable_lifts}.
\end{remark}

\begin{notation}
 Given a morphism
 $
  \alpha: (A \stackrel{\gamma_A}{\longrightarrow}\Omega_A \stackrel{\rho_A}{\longleftarrow} R_A) \longrightarrow (B \stackrel{\gamma_B}{\longrightarrow}\Omega_B \stackrel{\rho_B}{\longleftarrow} R_B)
 $
 in $\ACIC( \PC )$,
 we denote by
 \[
  \overline{\alpha}: (A \stackrel{\gamma_A}{\longrightarrow}\Omega_A \stackrel{\rho_A}{\longleftarrow} R_A) \longrightarrow (B \stackrel{\gamma_B}{\longrightarrow}\Omega_B \stackrel{\rho_B}{\longleftarrow} R_B)
 \]
 the corresponding morphism in $\CIC(\PC)$.
\end{notation}

\begin{construction}
 We construct a full and faithful additive functor 
 \[
 \emb: \PC \rightarrow \CIC(\PC)
 \]
 that identifies $\PC$ as a full subcategory of $\CIC(\PC)$.
 On objects, we set
 \[
  A \mapsto (A \stackrel{\id}{\rightarrow} A \leftarrow 0)
 \]
 and on morphisms, we set
 \[
  (A \stackrel{\alpha}{\longrightarrow} B) \mapsto (\emb(A) \stackrel{\overline{\alpha}}{\longrightarrow} \emb(B)).
 \]
\end{construction}
\begin{proof}[Correctness of the construction]
 Syzygies in $\SyzC(A \stackrel{\id}{\rightarrow} A \leftarrow 0)$ 
 are of the form $S \stackrel{0}{\rightarrow} A$.
\end{proof}

The objects in $\PC$ yield
a convenient way to cover the objects in $\CIC( \PC )$.

\begin{lemma}\label{lemma:convenient_covers}
 Identities of objects in $\PC$ yield well-defined epimorphisms in $\CIC( \PC )$:
 \[
  \emb(A) \stackrel{\overline{\id_A}}{\longrightarrow} (A \stackrel{\gamma_A}{\longrightarrow}\Omega_A \stackrel{\rho_A}{\longleftarrow} R_A).
 \]
\end{lemma}
\begin{proof}
Well-definedness is trivial since syzygy witnesses can be given by zero morphisms.
Moreover, being an epimorphism follows from Lemma \ref{lemma:epis}.
\end{proof}

\begin{lemma}\label{lemma:epis}
  Every morphism in $\CIC(\PC)$ of the form
  \[
   \overline{\id_B}: (B \stackrel{\gamma_B}{\longrightarrow}\Omega_B \stackrel{\rho_B}{\longleftarrow} R_B) \longrightarrow (B \stackrel{\gamma_B'}{\longrightarrow}\Omega_B' \stackrel{\rho_B'}{\longleftarrow} R_B')
  \]
  is an epimorphism.
 \end{lemma}
 \begin{proof}
  Given a morphism
  \[
   \overline{\tau}: (B \stackrel{\gamma_B'}{\longrightarrow}\Omega_B' \stackrel{\rho_B'}{\longleftarrow} R_B') \longrightarrow (T \stackrel{\gamma_T}{\longrightarrow}\Omega_T \stackrel{\rho_T}{\longleftarrow} R_T)
  \]
  such that $\overline{\id_B} \cdot \overline{\tau} = 0$,
  this means that there exists $\zeta: B \rightarrow R_T$
  such that
  \[
   \zeta \cdot \rho_T = \id_B \cdot \tau \cdot \gamma_T = \tau \cdot \gamma_T,
  \]
  which implies that $\overline{\tau}$ is already zero.
 \end{proof}
 
\subsection{Cokernels}

As a first main feature of $\CIC(\PC)$, we show how to construct
cokernels.

\begin{construction}[Cokernels]\label{construction:cokernel}
 Given a morphism
 \[
  \overline{\alpha}: (A \stackrel{\gamma_A}{\longrightarrow}\Omega_A \stackrel{\rho_A}{\longleftarrow} R_A) \longrightarrow (B \stackrel{\gamma_B}{\longrightarrow}\Omega_B \stackrel{\rho_B}{\longleftarrow} R_B)
 \]
 in $\CIC(\PC)$, the following diagram
 shows us how to construct its cokernel projection along with the universal property:
 \begin{equation}\label{equation:cokernel_diagram}
    \begin{tikzpicture}[label/.style={postaction={
            decorate,
            decoration={markings, mark=at position .5 with \node #1;}}}, baseline = (A),
            mylabel/.style={thick, draw=none, align=center, minimum width=0.5cm, minimum height=0.5cm,fill=white}]
          \coordinate (r) at (5.7,0);
          \coordinate (u) at (0,2);
          \node (A) {$(A \stackrel{\gamma_A}{\longrightarrow}\Omega_A \stackrel{\rho_A}{\longleftarrow} R_A)$};
          \node (B) at ($(A)+0.8*(r)$) {$(B \stackrel{\gamma_B}{\longrightarrow}\Omega_B \stackrel{\rho_B}{\longleftarrow} R_B)$};
          \node (C) at ($(B) + (r) + (u)$) {$(B \stackrel{\gamma_B}{\longrightarrow}\Omega_B \stackrel{\pmatcol{\rho_B}{\alpha \cdot \gamma_B}}{\longleftarrow} R_B \oplus A)$};
          \node (T) at ($(B) + (r) - (u)$) {$(T \stackrel{\gamma_T}{\longrightarrow}\Omega_T \stackrel{\rho_T}{\longleftarrow} R_T)$.};
          \draw[->,thick] (A) --node[above]{$\overline{\alpha}$} (B);
          \draw[->,thick] (B) --node[below,yshift=-0.1em,xshift=-1em]{$\overline{\tau}$} (T);
          \draw[->,thick] (B) --node[above]{$\overline{\id_B}$} (C);
          \draw[->,thick, dashed] (C) --node[right]{$\overline{\tau}$} (T);
          \node (Ae) at ($(A)+(-1.4,-0.15)$) {};
          \node (Te) at ($(T)-(-1.05,0.15)$) {};
          \draw[bend right,->,thick,dotted,label={[above]{$\zeta$}},out=360-25,in=360-155] (Ae) to (Te);
    \end{tikzpicture}
 \end{equation}
 How to read this diagram: the solid arrow pointing up right is the cokernel projection,
 the solid arrow pointing down right is a test morphism for the universal property of the cokernel,
 and the dashed arrow pointing straight down is the morphism induced by the universal property.
 The dotted arrow labeled with $\zeta$ is a witness for the composition $\overline{\alpha} \cdot \overline{\tau}$ in $\CIC(\PC)$ being zero,
 i.e., it denotes a morphism $\zeta: A \rightarrow R_T$ such that $\zeta \cdot \rho_T = \alpha \cdot \tau \cdot \gamma_T$.
\end{construction}
\begin{proof}[Correctness of the construction]
 Clearly, the morphism $\overline{\id_B}$ for the cokernel projection is well-defined,
 since syzygy witnesses 
 of objects in $\SyzC(B \stackrel{\gamma_B}{\longrightarrow}\Omega_B \stackrel{\rho_B}{\longleftarrow} R_B)$
 can simply be extended by the natural inclusion morphism $R_B \rightarrow R_B \oplus A$.
 Composing $\overline{\alpha}$ with the cokernel projection yields zero
  since we can take the natural inclusion $A \rightarrow R_B \oplus A$ as a witness for being zero.
 Next, we have to check well-definedness of the cokernel induced morphism.
 Given a syzygy 
 \begin{center}
          \begin{tikzpicture}[label/.style={postaction={
            decorate,
            decoration={markings, mark=at position .5 with \node #1;}}}, baseline = (D),
            mylabel/.style={thick, draw=none, align=center, minimum width=0.5cm, minimum height=0.5cm,fill=white}]
              \coordinate (r) at (2.5,0);
              \coordinate (u) at (0,1.2);
              \node (A) {$B$};
              \node (B) at ($(A)+1*(r)$) {$\Omega_B$};
              \node (C) at ($(B) +(r)$) {$R_B \oplus A$,};
              \node (D) at ($(A) + (u)$) {$S$};
              \draw[->,thick] (A) -- node[below]{$\gamma_B$} (B);
              \draw[->,thick] (C) -- node[below]{$\pmatcol{\rho_B}{\alpha \cdot \gamma_B}$} (B);
              \draw[->,thick] (D) -- node[left]{$\sigma$} (A);
              \draw[->,thick,dashed] (D) -- node[above,yshift=0.3em]{$\pmatrow{\lambda_1}{\lambda_2}$} (C);
          \end{tikzpicture}
 \end{center}
 we can construct another one:
 \begin{center}
          \begin{tikzpicture}[label/.style={postaction={
            decorate,
            decoration={markings, mark=at position .5 with \node #1;}}}, baseline = (D),
            mylabel/.style={thick, draw=none, align=center, minimum width=0.5cm, minimum height=0.5cm,fill=white}]
              \coordinate (r) at (2.5,0);
              \coordinate (u) at (0,1.2);
              \node (A) {$B$};
              \node (B) at ($(A)+1*(r)$) {$\Omega_B$};
              \node (C) at ($(B) +(r)$) {$R_B$.};
              \node (D) at ($(A) + (u)$) {$S$};
              \draw[->,thick] (A) -- node[below]{$\gamma_B$} (B);
              \draw[->,thick] (C) -- node[below]{$\rho_B$} (B);
              \draw[->,thick] (D) -- node[left]{$\sigma - \lambda_2 \cdot \alpha$} (A);
              \draw[->,thick,dashed] (D) -- node[above,yshift=0.3em]{$\lambda_1$} (C);
          \end{tikzpicture}
 \end{center}
 Now, applying the well-definedness property of the test morphism, 
 we obtain the syzygy
 \begin{equation}\label{equation:1}
          \begin{tikzpicture}[label/.style={postaction={
            decorate,
            decoration={markings, mark=at position .5 with \node #1;}}}, baseline = (A),
            mylabel/.style={thick, draw=none, align=center, minimum width=0.5cm, minimum height=0.5cm,fill=white}]
              \coordinate (r) at (2.5,0);
              \coordinate (u) at (0,1.2);
              \node (A) {$B$};
              \node (B) at ($(A)+1*(r)$) {};
              \node (C) at ($(B) +(r)$) {};
              \node (D) at ($(A) + (u)$) {$S$};
              \node (A2) at ($(A) - (u)$) {$T$};
              \node (B2) at ($(B) - (u)$) {$\Omega_T$};
              \node (C2) at ($(C) - (u)$) {$R_T$.};
              \draw[->,thick] (A2) -- node[below]{$\gamma_T$} (B2);
              \draw[->,thick] (C2) -- node[below]{$\rho_T$} (B2);
              \draw[->,thick] (D) --node[left]{$\sigma - \lambda_2 \cdot \alpha$}  (A);
              \draw[->,thick,dashed] (D) --node[right,xshift=0.5em,yshift=0.3em]{$\lambda_3$} (C2);
              \draw[->,thick] (A) --node[left]{$\tau$} (A2);
          \end{tikzpicture}
\end{equation}
We deduce that
  \begin{center}
          \begin{tikzpicture}[label/.style={postaction={
            decorate,
            decoration={markings, mark=at position .5 with \node #1;}}}, baseline = (D),
            mylabel/.style={thick, draw=none, align=center, minimum width=0.5cm, minimum height=0.5cm,fill=white}]
              \coordinate (r) at (2.5,0);
              \coordinate (u) at (0,1.2);
              \node (A) {$B$};
              \node (B) at ($(A)+1*(r)$) {};
              \node (C) at ($(B) +(r)$) {};
              \node (D) at ($(A) + (u)$) {$S$};
              \node (A2) at ($(A) - (u)$) {$T$};
              \node (B2) at ($(B) - (u)$) {$\Omega_T$};
              \node (C2) at ($(C) - (u)$) {$R_T$};
              \draw[->,thick] (A2) -- node[below]{$\gamma_T$} (B2);
              \draw[->,thick] (C2) -- node[below]{$\rho_T$} (B2);
              \draw[->,thick] (D) --node[left]{$\sigma$}  (A);
              \draw[->,thick,dashed] (D) --node[right,xshift=0.5em,yshift=0.3em]{$\lambda_2 \cdot \zeta + \lambda_3$} (C2);
              \draw[->,thick] (A) --node[left]{$\tau$} (A2);
          \end{tikzpicture}
  \end{center}
  is a syzygy by computing
  \begin{align*}
   \sigma \cdot \tau \cdot \gamma_T &= \lambda_2 \cdot \alpha \cdot \tau \cdot \gamma_T + \lambda_3 \cdot \rho_T &\text{using \eqref{equation:1} }\\
   &= \lambda_2 \cdot \zeta \cdot \rho_T + \lambda_3 \cdot \rho_T &\text{using the defining equation of $\zeta$}\\
   & = (\lambda_2 \cdot \zeta + \lambda_3) \cdot \rho_T. &
  \end{align*}
  Thus, the cokernel induced morphism is well-defined
  and it clearly renders the triangle in \eqref{equation:cokernel_diagram} commutative.
  For the uniqueness of the induced morphism, it suffices to check that the cokernel projection is an epimorphism,
  which is the content of Lemma \ref{lemma:epis}.
\end{proof}

\subsection{Lifts along monomorphisms}

We show that every monomorphism in $\CIC( \PC )$
is the kernel of its cokernel by means of the following construction.

\begin{construction}[Lifts along monomorphisms]\label{construction:lift_along_mono}
The following diagram shows us how to construct a lift along 
a given monomorphism
\[
  \overline{\alpha}: (A \stackrel{\gamma_A}{\longrightarrow}\Omega_A \stackrel{\rho_A}{\longleftarrow} R_A) \longrightarrow (B \stackrel{\gamma_B}{\longrightarrow}\Omega_B \stackrel{\rho_B}{\longleftarrow} R_B)
 \]
 in $\CIC(\PC)$
 for a given test morphism:
 \begin{center}
    \begin{tikzpicture}[label/.style={postaction={
          decorate,
          decoration={markings, mark=at position .5 with \node #1;}},
          mylabel/.style={thick, draw=black, align=center, minimum width=0.5cm, minimum height=0.5cm,fill=white}}]
          \coordinate (r) at (5.5,0);
          \coordinate (u) at (0,2);
          \node (A) {$(B \stackrel{\gamma_B}{\longrightarrow}\Omega_B \stackrel{\rho_B}{\longleftarrow} R_B)$};
          \node (B) at ($(A)+1.09*(r)$) {$(B \stackrel{\gamma_B}{\longrightarrow}\Omega_B \smash{\stackrel{\begin{pmatrix}\rho_B \\ \alpha \cdot \gamma_B\end{pmatrix}}{\longleftarrow}} R_B \oplus A)$.};
          \node (K) at ($(A) - 1*(r) + (u)$) {$(A \stackrel{\gamma_A}{\longrightarrow}\Omega_A \stackrel{\rho_A}{\longleftarrow} R_A)$};
          \node (T) at ($(A) - 1*(r) - (u)$) {$(T \stackrel{\gamma_T}{\longrightarrow}\Omega_T \stackrel{\rho_T}{\longleftarrow} R_T)$};
          \draw[->,thick] (A) --node[above]{$\overline{\id_B}$} (B);
          \draw[->,thick] (T) --node[below,yshift=-0.2em]{$\overline{\tau}$} (A);
          \draw[->,thick] (K) --node[above,xshift=0em,yshift=0.3em]{$\overline{\alpha}$} (A);
          \draw[->,thick,dashed,label={[mylabel]{$\overline{\zeta_2}$}}] (T) -- (K);
          
          \node (Be) at ($(B)-(-1.7,0.25)$) {};
          \node (Te) at ($(T)+(-1.5,-0.25)$) {};
          \draw[bend right,->,thick,dotted,label={[mylabel]{$\pmatrow{\zeta_1}{\zeta_2}$}},out=360-25,in=360-155] (Te) to (Be);
    \end{tikzpicture}
 \end{center}
 How to read this diagram: the solid horizontal arrow is the cokernel projection of 
 our monomorphism $\overline{\alpha}$ (see Construction \ref{construction:cokernel}).
 The dotted arrow is a witness for the composition of the test morphism $\overline{\tau}$ with the cokernel projection being zero, i.e.,
 the equation
 \begin{equation}\label{equation:lift_along_mono}
  \tau \cdot \gamma_B = \zeta_1 \cdot \rho_B + \zeta_2 \cdot \alpha \cdot \gamma_B
 \end{equation}
 holds.
 The upwards pointing dashed arrow is the desired lift.
\end{construction}
\begin{proof}[Correctness of the construction]
 First, we show that $\overline{\zeta_2}$ is well-defined.
 Given a syzygy
 \begin{center}
          \begin{tikzpicture}[label/.style={postaction={
            decorate,
            decoration={markings, mark=at position .5 with \node #1;}}}, baseline = (D),
            mylabel/.style={thick, draw=none, align=center, minimum width=0.5cm, minimum height=0.5cm,fill=white}]
              \coordinate (r) at (2.5,0);
              \coordinate (u) at (0,1.2);
              \node (A) {$T$};
              \node (B) at ($(A)+1*(r)$) {$\Omega_T$};
              \node (C) at ($(B) +(r)$) {$R_T$,};
              \node (D) at ($(A) + (u)$) {$S$};
              \draw[->,thick] (A) -- node[below]{$\gamma_T$} (B);
              \draw[->,thick] (C) -- node[below]{$\rho_T$} (B);
              \draw[->,thick] (D) -- node[left]{$\sigma$} (A);
              \draw[->,thick,dashed] (D) -- node[above]{$\lambda$} (C);
          \end{tikzpicture}
 \end{center}
 we can use the fact that $\overline{\tau}$ satisfies the well-definedness property
 in order to get a syzygy
 \begin{center}
          \begin{tikzpicture}[label/.style={postaction={
            decorate,
            decoration={markings, mark=at position .5 with \node #1;}}}, baseline = (D),
            mylabel/.style={thick, draw=none, align=center, minimum width=0.5cm, minimum height=0.5cm,fill=white}]
              \coordinate (r) at (2.5,0);
              \coordinate (u) at (0,1.2);
              \node (A) {$B$};
              \node (B) at ($(A)+1*(r)$) {$\Omega_B$};
              \node (C) at ($(B) +(r)$) {$R_B$.};
              \node (D) at ($(A) + (u)$) {$S$};
              \draw[->,thick] (A) -- node[below]{$\gamma_B$} (B);
              \draw[->,thick] (C) -- node[below]{$\rho_B$} (B);
              \draw[->,thick] (D) -- node[left]{$\sigma \cdot \tau$} (A);
              \draw[->,thick,dashed] (D) -- node[above]{$\lambda'$} (C);
          \end{tikzpicture}
 \end{center}
 Using \eqref{equation:lift_along_mono},
 we can construct another syzygy
 \begin{center}
          \begin{tikzpicture}[label/.style={postaction={
            decorate,
            decoration={markings, mark=at position .5 with \node #1;}}}, baseline = (D),
            mylabel/.style={thick, draw=none, align=center, minimum width=0.5cm, minimum height=0.5cm,fill=white}]
              \coordinate (r) at (2.5,0);
              \coordinate (u) at (0,1.2);
              \node (A) {$B$};
              \node (B) at ($(A)+1*(r)$) {$\Omega_B$};
              \node (C) at ($(B) +(r)$) {$R_B$};
              \node (D) at ($(A) + (u)$) {$S$};
              \draw[->,thick] (A) -- node[below]{$\gamma_B$} (B);
              \draw[->,thick] (C) -- node[below]{$\rho_B$} (B);
              \draw[->,thick] (D) -- node[left]{$\sigma \cdot \zeta_2 \cdot \alpha$} (A);
              \draw[->,thick,dashed] (D) -- node[above,xshift=0.5em,yshift=0.3em]{$\lambda' - \sigma \cdot \zeta_1$} (C);
          \end{tikzpicture}
 \end{center}
 whose syzygy witness can also be interpreted as a witness for the composition of
 \[
 \emb(S) \stackrel{\overline{\sigma \cdot \zeta_2}}{\longrightarrow}
 (A \stackrel{\gamma_A}{\longrightarrow}\Omega_A \stackrel{\rho_A}{\longleftarrow} R_A) \stackrel{\overline{\alpha}}{\longrightarrow} (B \stackrel{\gamma_B}{\longrightarrow}\Omega_B \stackrel{\rho_B}{\longleftarrow} R_B)\]
 in $\CIC(\PC)$ being zero.
 Since $\overline{\alpha}$ is a monomorphism, this implies $\overline{\sigma \cdot \zeta_2} = 0$,
 and so we get our desired syzygy:
 \begin{center}
          \begin{tikzpicture}[label/.style={postaction={
            decorate,
            decoration={markings, mark=at position .5 with \node #1;}}}, baseline = (D),
            mylabel/.style={thick, draw=none, align=center, minimum width=0.5cm, minimum height=0.5cm,fill=white}]
              \coordinate (r) at (2.5,0);
              \coordinate (u) at (0,1.2);
              \node (A) {$A$};
              \node (B) at ($(A)+1*(r)$) {$\Omega_A$};
              \node (C) at ($(B) +(r)$) {$R_A$.};
              \node (D) at ($(A) + (u)$) {$S$};
              \draw[->,thick] (A) -- node[below]{$\gamma_A$} (B);
              \draw[->,thick] (C) -- node[below]{$\rho_A$} (B);
              \draw[->,thick] (D) -- node[left]{$\sigma \cdot \zeta_2$} (A);
              \draw[->,thick, dashed] (D) -- (C);
          \end{tikzpicture}
 \end{center}
Furthermore, $\overline{\zeta_2} \cdot \overline{\alpha} = \overline{\tau}$ since
$\eqref{equation:lift_along_mono}$ may be rearranged as
 \[
  ( \zeta_2 \cdot \alpha - \tau ) \cdot \gamma_B = - \zeta_1 \cdot \rho_B.\qedhere
 \]
\end{proof}

\begin{corollary}\label{corollary:epi-mono_iso}
 Every morphism in $\CIC(\PC)$ that is both mono and epi is an isomorphism.
\end{corollary}
\begin{proof}
  By Construction \ref{construction:lift_along_mono} every mono in $\CIC( \PC )$ is a normal mono.
\end{proof}

\begin{remark}
 If $\PC$ has decidable syzygy inclusion,
 then we can decide whether a given morphism
 is a monomorphism by the following Lemma \ref{lemma:decide_mono}.
 In particular, we can check the assumption on the input in Construction \ref{construction:lift_along_mono}.
\end{remark}

\begin{lemma}\label{lemma:decide_mono}
  A morphism
 \[
  \overline{\alpha}: (A \stackrel{\gamma_A}{\longrightarrow}\Omega_A \stackrel{\rho_A}{\longleftarrow} R_A) \longrightarrow (B \stackrel{\gamma_B}{\longrightarrow}\Omega_B \stackrel{\rho_B}{\longleftarrow} R_B)
 \]
 in $\CIC( \PC )$ is a monomorphism
 if and only if
 \[
  \SyzC(A \stackrel{\alpha \cdot \gamma_B}{\longrightarrow}\Omega_B \stackrel{\rho_B}{\longleftarrow} R_B)
  \subseteq
  \SyzC(A \stackrel{\gamma_A}{\longrightarrow}\Omega_A \stackrel{\rho_A}{\longleftarrow} R_A).
 \]
\end{lemma}
\begin{proof}
 If $\overline{\alpha}$ is a monomorphism,
 and $\sigma \in \SyzC(A \stackrel{\alpha \cdot \gamma_B}{\longrightarrow}\Omega_B \stackrel{\rho_B}{\longleftarrow} R_B)$,
 then the composite
 \[
  \emb( S ) \stackrel{\overline{\sigma}}{\longrightarrow}
  (A \stackrel{\gamma_A}{\longrightarrow}\Omega_A \stackrel{\rho_A}{\longleftarrow} R_A) \stackrel{\overline{\alpha}}{\longrightarrow} (B \stackrel{\gamma_B}{\longrightarrow}\Omega_B \stackrel{\rho_B}{\longleftarrow} R_B)
 \]
 is zero, and thus $\overline{\sigma}$ is zero,
 which implies $\sigma \in \SyzC(A \stackrel{\gamma_A}{\longrightarrow}\Omega_A \stackrel{\rho_A}{\longleftarrow} R_A)$.
 
 Conversely,
 we can test being a monomorphism on compositions of the form
 \[
  \emb( S ) \stackrel{\overline{\sigma}}{\longrightarrow}
  (A \stackrel{\gamma_A}{\longrightarrow}\Omega_A \stackrel{\rho_A}{\longleftarrow} R_A) \stackrel{\overline{\alpha}}{\longrightarrow} (B \stackrel{\gamma_B}{\longrightarrow}\Omega_B \stackrel{\rho_B}{\longleftarrow} R_B)
 \]
 that yield zero due to Lemma \ref{lemma:convenient_covers}.
 But then, $\sigma \in \SyzC(A \stackrel{\alpha \cdot \gamma_B}{\longrightarrow}\Omega_B \stackrel{\rho_B}{\longleftarrow} R_B)$,
 which by assumption implies $\sigma \in \SyzC(A \stackrel{\gamma_A}{\longrightarrow}\Omega_A \stackrel{\rho_A}{\longleftarrow} R_A)$,
 which is equivalent to $\overline{\sigma}$ being zero.
\end{proof}

\subsection{Universal epi-mono factorizations}

As another decisive feature, $\CIC(\PC)$ admits universal epi-mono factorizations, i.e., essentially unique epi-mono factorizations,
and thus in particular images.

\begin{remark}
  Since we proved in Construction \ref{construction:lift_along_mono}
  that every mono in $\CIC( \PC )$ is a normal mono,
  the theory of factorizations
  as it is presented in \cite[Section 2]{FK72}
  implies that it suffices to prove that every morphism in $\CIC( \PC )$
  admits a factorization into a mono and an epi
  in order to conclude this factorization is already universal.
  Nevertheless, in Construction \ref{construction:epi_mono_factorization},
  we will make the universality of the epi-mono factorization explicit,
  since from the perspective of a computer implementation, it is helpful to have
  concrete formulas at hand.
\end{remark}

\begin{construction}[Universal epi-mono factorization]\label{construction:epi_mono_factorization}
 Given a morphism
 \[\overline{\alpha}: (A \stackrel{\gamma_A}{\longrightarrow}\Omega_A \stackrel{\rho_A}{\longleftarrow} R_A) \longrightarrow (B \stackrel{\gamma_B}{\longrightarrow}\Omega_B \stackrel{\rho_B}{\longleftarrow} R_B)\]
 in $\CIC(\PC)$,
 the following diagram shows us how to construct its universal epi-mono factorization
 along with its universal property:
 \begin{center}
    \begin{tikzpicture}[label/.style={postaction={
          decorate,
          decoration={markings, mark=at position .5 with \node #1;}},
          mylabel/.style={thick, draw=none, align=center, minimum width=0.5cm, minimum height=0.5cm,fill=white}}]
          \coordinate (r) at (6,0);
          \coordinate (u) at (0,2);
          \node (A) {$(A \stackrel{\gamma_A}{\longrightarrow}\Omega_A \stackrel{\rho_A}{\longleftarrow} R_A)$};
          \node (B) at ($(A)+2*(r)$) {$(B \stackrel{\gamma_B}{\longrightarrow}\Omega_B \stackrel{\rho_B}{\longleftarrow} R_B)$};
          \node (C) at ($(A) + (r) - (u)$) {$(A \stackrel{\alpha \cdot \gamma_B}{\longrightarrow}\Omega_B \stackrel{\rho_B}{\longleftarrow} R_B)$};
          \node (T) at ($(A) + (r) - 2*(u)$) {$(T \stackrel{\gamma_T}{\longrightarrow}\Omega_T \stackrel{\rho_T}{\longleftarrow} R_T)$};
          \draw[->,thick] (A) --node[above]{$\overline{\alpha}$} (B);
          \draw[->,thick] (A) --node[above,yshift=0em,xshift=0em]{$\overline{\id_A}$} (C);
          \draw[->,thick] (C) --node[above]{$\overline{\alpha}$} (B);
          \draw[bend right,->,thick,label={[{below,yshift=-0.2em}]{$\overline{\tau_1}$}},out=360-25,in=360-155] (A) to (T);
          \draw[bend right,->,thick,label={[{below,yshift=-0.2em}]{$\overline{\tau_2}$}},out=-25,in=210] (T) to (B);
          \draw[->,thick, dashed] (C) --node[right]{$\overline{\tau_1}$} (T);
    \end{tikzpicture}
 \end{center}
 How to read this diagram: the universal epi-mono factorization of $\overline{\alpha}$ is given by the
 upper triangle.
 Furthermore, if $\overline{\tau_1}$ and $\overline{\tau_2}$ form another epi-mono factorization of $\overline{\alpha}$,
 then the dashed vertical arrow is the isomorphism induced by its universal property.
\end{construction}
\begin{proof}[Correctness of the construction]
 The map 
 \[
  \SyzC(A \stackrel{\gamma_A}{\longrightarrow}\Omega_A \stackrel{\rho_A}{\longleftarrow} R_A)
  \rightarrow
  \SyzC(A \stackrel{\alpha \cdot \gamma_B}{\longrightarrow}\Omega_B \stackrel{\rho_B}{\longleftarrow} R_B): (S \rightarrow A) \mapsto (S \rightarrow A)
 \]
 is well-defined since
 \[
  \SyzC(A \stackrel{\gamma_A}{\longrightarrow}\Omega_A \stackrel{\rho_A}{\longleftarrow} R_A)
  \rightarrow
  \SyzC(B \stackrel{\gamma_B}{\longrightarrow}\Omega_B \stackrel{\rho_B}{\longleftarrow} R_B): 
  (S \stackrel{\sigma}{\rightarrow} A) \mapsto (S \stackrel{\sigma \cdot \alpha}\longrightarrow B)
 \]
 is well-defined.
 Furthermore, the map
 \[
  \SyzC(A \stackrel{\alpha \cdot \gamma_B}{\longrightarrow}\Omega_B \stackrel{\rho_B}{\longleftarrow} R_B)
  \rightarrow
  \SyzC(B \stackrel{\gamma_B}{\longrightarrow}\Omega_B \stackrel{\rho_B}{\longleftarrow} R_B):
  (S \stackrel{\sigma}{\rightarrow} A) \mapsto (S \stackrel{\sigma \cdot \alpha}\longrightarrow B)
 \]
 is always well-defined.
 Thus, we verified that the candidate for the universal epi-mono factorization
 consists of well-defined morphisms. Lemma \ref{lemma:epis} shows that $\overline{\id_A}$ is an epimorphism,
 and Lemma \ref{lemma:decide_mono} proves that we really have an epi-mono factorization.
 
 To check the well-definedness property of the induced morphism,
 we start with a syzygy
 \begin{center}
          \begin{tikzpicture}[label/.style={postaction={
            decorate,
            decoration={markings, mark=at position .5 with \node #1;}}}, baseline = (D),
            mylabel/.style={thick, draw=none, align=center, minimum width=0.5cm, minimum height=0.5cm,fill=white}]
              \coordinate (r) at (2.5,0);
              \coordinate (u) at (0,1.2);
              \node (A) {$A$};
              \node (B) at ($(A)+1*(r)$) {$\Omega_B$};
              \node (C) at ($(B) +(r)$) {$R_B$};
              \node (D) at ($(A) + (u)$) {$S$};
              \draw[->,thick] (A) -- node[below]{$\alpha \cdot \gamma_B$} (B);
              \draw[->,thick] (C) -- node[below]{$\rho_B$} (B);
              \draw[->,thick] (D) -- node[left]{$\sigma$} (A);
              \draw[->,thick,dashed] (D) -- node[above]{$\lambda$} (C);
          \end{tikzpicture}
 \end{center}
 and see that the syzygy witness $\lambda$ can be interpreted as a witness for the composition
 \[
  \emb(S) \stackrel{\overline{\sigma}}{\longrightarrow} 
  (A \stackrel{\gamma_A}{\longrightarrow}\Omega_A \stackrel{\rho_A}{\longleftarrow} R_A) \stackrel{\overline{\alpha}}{\longrightarrow} 
  (B \stackrel{\gamma_B}{\longrightarrow}\Omega_B \stackrel{\rho_B}{\longleftarrow} R_B)
 \]
 being zero. From
 \begin{align*}
  0 = \overline{\sigma} \cdot \overline{\alpha} = \overline{\sigma} \cdot \overline{\tau_1} \cdot \overline{\tau_2}
 \end{align*}
 and $\overline{\tau_2}$ being a monomorphism we conclude $\overline{\sigma} \cdot \overline{\tau_1} = 0$,
 which gives us the desired syzygy:
 \begin{center}
          \begin{tikzpicture}[label/.style={postaction={
            decorate,
            decoration={markings, mark=at position .5 with \node #1;}}}, baseline = (D),
            mylabel/.style={thick, draw=none, align=center, minimum width=0.5cm, minimum height=0.5cm,fill=white}]
              \coordinate (r) at (2.5,0);
              \coordinate (u) at (0,1.2);
              \node (A) {$T$};
              \node (B) at ($(A)+1*(r)$) {$\Omega_T$};
              \node (C) at ($(B) +(r)$) {$R_T$.};
              \node (D) at ($(A) + (u)$) {$S$};
              \draw[->,thick] (A) -- node[below]{$\gamma_T$} (B);
              \draw[->,thick] (C) -- node[below]{$\rho_T$} (B);
              \draw[->,thick] (D) -- node[left]{$\sigma \cdot \tau_1$} (A);
              \draw[->,thick, dashed] (D) -- (C);
          \end{tikzpicture}
 \end{center}
 Thus, the induced morphism is well-defined. It is easy to check
 that it renders the whole epi-mono factorization diagram commutative:
 the lower left triangle commutes already in $\ACIC( \PC )$,
 and from this, the commutativity of the lower right triangle is implied.
 
 Last, since the induced morphism is an epimorphism and a monomorphism,
 Corollary \ref{corollary:epi-mono_iso} proves that it is an isomorphism.
 Thus, we have successfully constructed a universal epi-mono factorization.
\end{proof}

\subsection{Colifts along epimorphisms}
The category $\CIC(\PC)$ does not necessarily have kernels (see Theorem \ref{theorem:abelian_case}).
Thus, it does not make sense to ask for every epimorphism to be the cokernel
of its kernel. However, the following construction serves as an
appropriate substitute.

\begin{construction}[Colifts along epimorphisms]\label{construction:colift_along_epi}
 Let
 \[
 \overline{\alpha}: (A \stackrel{\gamma_A}{\longrightarrow}\Omega_A \stackrel{\rho_A}{\longleftarrow} R_A) \longrightarrow (B \stackrel{\gamma_B}{\longrightarrow}\Omega_B \stackrel{\rho_B}{\longleftarrow} R_B)
 \]
 be an epi
 in $\CIC(\PC)$.
 Then, its cokernel projection is the zero morphism.
 Using the explicit construction of the cokernel projection
 in Construction \ref{construction:cokernel}, this means that
 there exists a morphism
 \[
  \pmatrow{\zeta_1}{\zeta_2}: B \longrightarrow R_B \oplus A
 \]
 such that
 \begin{equation}\label{equation:colift_along_epi}
  \gamma_B = \zeta_1 \cdot \rho_B + \zeta_2 \cdot \alpha \cdot \gamma_B.
 \end{equation}
 The following diagram shows us how to construct
 a colift along the epimorphism $\overline{\alpha}$
 \begin{equation}\label{equation:colift_along_epi_diagram}
          \begin{tikzpicture}[label/.style={postaction={
            decorate,
            decoration={markings, mark=at position .5 with \node #1;}}}, baseline = (base),
            mylabel/.style={thick, draw=none, align=center, minimum width=0.5cm, minimum height=0.5cm,fill=white}]
              \coordinate (r) at (5,0);
              \coordinate (d) at (0,-2);
              \node (A) {$(A \stackrel{\gamma_A}{\longrightarrow}\Omega_A \stackrel{\rho_A}{\longleftarrow} R_A)$};
              \node (B) at ($(A)+(r)$) {$(B \stackrel{\gamma_B}{\longrightarrow}\Omega_B \stackrel{\rho_B}{\longleftarrow} R_B)$};
              \node (T) at ($(B) +(d)$) {$(T \stackrel{\gamma_T}{\longrightarrow}\Omega_T \stackrel{\rho_T}{\longleftarrow} R_T)$};
              \node (base) at ($(A) + 0.5*(d)$) {};
              \draw[->,thick] (A) -- node[above]{$\overline{\alpha}$} (B);
              \draw[->,thick] (A) -- node[below]{$\overline{\tau}$} (T);
              \draw[->,thick,dashed] (B) -- node[right]{$\overline{\zeta_2 \cdot \tau}$} (T);
          \end{tikzpicture}
 \end{equation}
 for a given test morphism $\overline{\tau}$, where
 test morphism means that $\overline{\tau}$ satisfies the following property:
 whenever we have a morphism $\overline{\kappa}$ in $\CIC(\PC)$ such that
 $\overline{\kappa} \cdot \overline{\alpha} = 0$, we also have
 $\overline{\kappa} \cdot \overline{\tau} = 0$.
\end{construction}
\begin{remark}
 If $\PC$ has decidable syzygy inclusion,
 then we can decide whether a given $\overline{\tau}$ yields a test morphism:
 indeed, using Lemma \ref{lemma:convenient_covers},
 this is the case if and only if
 \[
  \SyzC( A \stackrel{\alpha \cdot \gamma_B}{\longrightarrow} \Omega_B \stackrel{\rho_B}{\longleftarrow} R_B )
  \subseteq
  \SyzC( A \stackrel{\tau \cdot \gamma_T}{\longrightarrow} \Omega_T \stackrel{\rho_T}{\longleftarrow} R_T ).
 \]

\end{remark}

\begin{proof}[Correctness of the construction]
 First, we show that the colift $\overline{\zeta_2 \cdot \tau}$ satisfies the well-definedness property.
 Given a syzygy
 \begin{center}
          \begin{tikzpicture}[label/.style={postaction={
            decorate,
            decoration={markings, mark=at position .5 with \node #1;}}}, baseline = (D),
            mylabel/.style={thick, draw=none, align=center, minimum width=0.5cm, minimum height=0.5cm,fill=white}]
              \coordinate (r) at (2.5,0);
              \coordinate (u) at (0,1.2);
              \node (A) {$B$};
              \node (B) at ($(A)+1*(r)$) {$\Omega_B$};
              \node (C) at ($(B) +(r)$) {$R_B$};
              \node (D) at ($(A) + (u)$) {$S$};
              \draw[->,thick] (A) -- node[below]{$\gamma_B$} (B);
              \draw[->,thick] (C) -- node[below]{$\rho_B$} (B);
              \draw[->,thick] (D) -- node[left]{$\sigma$} (A);
              \draw[->,thick,dashed] (D) -- node[above,yshift=0.3em]{$\lambda$} (C);
          \end{tikzpicture}
 \end{center}
 we conclude by multiplying \eqref{equation:colift_along_epi} with $\sigma$ from the left that
 \begin{center}
          \begin{tikzpicture}[label/.style={postaction={
            decorate,
            decoration={markings, mark=at position .5 with \node #1;}}}, baseline = (D),
            mylabel/.style={thick, draw=none, align=center, minimum width=0.5cm, minimum height=0.5cm,fill=white}]
              \coordinate (r) at (2.5,0);
              \coordinate (u) at (0,1.2);
              \node (A) {$B$};
              \node (B) at ($(A)+1*(r)$) {$\Omega_B$};
              \node (C) at ($(B) +(r)$) {$R_B$};
              \node (D) at ($(A) + (u)$) {$S$};
              \draw[->,thick] (A) -- node[below]{$\gamma_B$} (B);
              \draw[->,thick] (C) -- node[below]{$\rho_B$} (B);
              \draw[->,thick] (D) -- node[left]{$\sigma \cdot \zeta_2 \cdot \alpha$} (A);
              \draw[->,thick,dashed] (D) -- node[above,yshift=0.3em]{$\lambda - \sigma \cdot \zeta_1$} (C);
          \end{tikzpicture}
 \end{center}
 is also a syzygy, whose syzygy witness can be interpreted as a witness for the composition
 \[
 \emb(S) \stackrel{\overline{\sigma \cdot \zeta_2}}{\longrightarrow}
 (A \stackrel{\gamma_A}{\longrightarrow}\Omega_A \stackrel{\rho_A}{\longleftarrow} R_A) \stackrel{\overline{\alpha}}{\longrightarrow} (B \stackrel{\gamma_B}{\longrightarrow}\Omega_B \stackrel{\rho_B}{\longleftarrow} R_B)\]
 in $\CIC(\PC)$ being zero.
 Since $\overline{\tau}$ is a test morphism, this implies that the composition
 \[
 \emb(S) \stackrel{\overline{\sigma \cdot \zeta_2}}{\longrightarrow}
 (A \stackrel{\gamma_A}{\longrightarrow}\Omega_A \stackrel{\rho_A}{\longleftarrow} R_A) \stackrel{\overline{\tau}}{\longrightarrow} (T \stackrel{\gamma_T}{\longrightarrow}\Omega_T \stackrel{\rho_T}{\longleftarrow} R_T)\]
 is zero as well, which gives us the desired well-definedness property.
 
 To show that $\overline{\zeta_2 \cdot \tau}$ is really a colift,
 we multiply \eqref{equation:colift_along_epi} with $\alpha$ from the left
 to see that the composition
 \[
 \emb(A) \stackrel{\overline{\alpha \cdot \zeta_2 - \id_A}}{\longrightarrow}
 (A \stackrel{\gamma_A}{\longrightarrow}\Omega_A \stackrel{\rho_A}{\longleftarrow} R_A) \stackrel{\overline{\alpha}}{\longrightarrow} (B \stackrel{\gamma_B}{\longrightarrow}\Omega_B \stackrel{\rho_B}{\longleftarrow} R_B)\]
 is zero. Since $\overline{\tau}$ is a test morphism,
 the composition
 \[
 \emb(A) \stackrel{\overline{\alpha \cdot \zeta_2 - \id_A}}{\longrightarrow}
 (A \stackrel{\gamma_A}{\longrightarrow}\Omega_A \stackrel{\rho_A}{\longleftarrow} R_A) \stackrel{\overline{\tau}}{\longrightarrow} (T \stackrel{\gamma_T}{\longrightarrow}\Omega_T \stackrel{\rho_T}{\longleftarrow} R_T)\]
 is also zero.
 If $\zeta_3$ denotes a witness for this composition being zero, then the equation
 \[
  (\alpha \cdot \zeta_2 \cdot \tau - \tau) \cdot \gamma_T = \zeta_3 \cdot \rho_T
 \]
 holds, which means that the diagram \eqref{equation:colift_along_epi_diagram} commutes.
\end{proof}

\section{The category $\CIC(\PC)$ as a subcategory of the category of modules}\label{section:modules}
Let $\PC$ be an additive category\footnote{To avoid set-theoretic issues, we assume $\PC$ to be a small category, i.e., its objects form a set.}.
We denote the category of contravariant additive functors from $\PC$ to the
category of abelian groups $\Ab$ by
$\Modr \PC$ and call it the category of \textbf{right $\PC$-modules}.

\begin{example}\label{example:functors_are_modules}
 An additive functor
 \[
  F: \Rows_R^{\op} \rightarrow \Ab
 \]
 is uniquely determined up to natural isomorphism by its restriction to the full subcategory of $\Rows_R^{\op}$ spanned by $R^{1 \times 1}$,
 since it respects direct sums.
 The image $F( R^{1 \times 1} )$ is an abelian group, 
 and the action of $F$ on morphisms encodes a left action of $R$ on $F( R^{1 \times 1} )$,
 giving it the structure of a left $R$-module.
 In this way, we get an equivalence of categories\footnote{
   At first glance, it might look confusing that \emph{right} $\Rows_R$-modules correspond to \emph{left} $R$-modules.
   If we regard $R$ as a category with a single object $\ast$ whose endomorphisms are given by the elements of $R$,
   then it is common to define the \emph{post}composition as ring multiplication.
   With this convention, functors $R \rightarrow \Ab$ which respect addition correspond to \emph{left} modules.
   But since we defined $\Rows_R$ as a subcategory of \emph{left} $R$-modules,
   we get a \emph{contravariant} functor $R \rightarrow \Rows_R: \ast \mapsto R^{1 \times 1}$.
 }
 \[
  R\Modl \simeq \Modr \Rows_R.
 \]
\end{example}

\begin{notation}
 We let
\[
 \PC \longrightarrow \Modr\PC: P \mapsto (-,P)
\]
denote the Yoneda embedding, where $(-,P)$ is shorthand notation for $\Hom_{\PC}(-,P)$.
\end{notation}

\begin{remark}[A short interlude on working with $\Modr\PC$]\label{remark:working_with_functor}
Since $\Modr\PC$ is a functor category, all limits and colimits
are computed pointwise, i.e., after evaluation at every object $A \in \PC$, see, e.g., \cite[Chapter V.3]{MLCWM}.
Since $\Modr\PC$ is abelian, the pointwise constructions
apply in particular to kernels, cokernels, and images.
Deciding whether a morphism in $\Modr\PC$ is mono/epi
can also be decided pointwise, since it
is equivalent to the kernel/cokernel being zero,
which can be decided pointwise.
For every $A \in \PC$ and $F \in \Modr\PC$, the Yoneda lemma states that 
a morphism
\[
 (-,A) \rightarrow F
\]
is uniquely determined by 
choosing an image $x \in F(A)$ of the element $\id_A \in (A,A)$, and every such choice is valid.
In particular, the Yoneda lemma
implies that the object $(-,A) \in \Modr\PC$
is projective in $\Modr\PC$.
\end{remark}

The goal of this section is to construct a functor
\[
 \overline{M}: \CIC( \PC ) \longrightarrow \Modr\PC.
\]
We proceed in several steps.

\begin{construction}\label{construction:functor_aux_mod}
 As a first step, 
 we are going to construct a functor
 \[
  M: \ACIC(\PC) \longrightarrow \Modr\PC
 \]
 on objects.
 From an object $(A \stackrel{\gamma_A}{\longrightarrow}\Omega_A \stackrel{\rho_A}{\longleftarrow} R_A)$ in $\ACIC( \PC )$,
 we obtain a short exact sequence
 \begin{center}
            \begin{tikzpicture}[label/.style={postaction={
              decorate,
              decoration={markings, mark=at position .5 with \node #1;}}},
              mylabel/.style={thick, draw=none, align=center, minimum width=0.5cm, minimum height=0.5cm,fill=white}]
                \coordinate (r) at (2.5,0);
                \coordinate (u) at (0,1.2);
                \node (A) {$0$};
                \node (B) at ($(A)+(r)$) {$\ker( \epsilon_A )$};
                \node (C) at ($(B)+(r)$) {$(-,A)$};
                \node (D) at ($(C)+(r)$) {$\frac{\im( -, \gamma_A )}{ \im( -, \rho_A )}$};
                \node (E) at ($(D)+(r)$) {$0$};
                \draw[->,thick] (A) -- (B);
                \draw[->,thick] (B) -- (C);
                \draw[->,thick] (C) -- node[above]{$\epsilon_A$} (D);
                \draw[->,thick] (D) -- (E);
            \end{tikzpicture}
  \end{center}
  in $\Modr \PC$, where $\epsilon_A$ is the composition of $(-,A) \rightarrow \im(-,\gamma_A)$,
  the coastriction to image of the morphism $(-,\gamma_A)$,
  with the projection $\im(-,\gamma_A) \rightarrow \frac{\im( -, \gamma_A )}{ \im( -, \rho_A )}$.
  We set
  \[
   M\big( (A \stackrel{\gamma_A}{\longrightarrow}\Omega_A \stackrel{\rho_A}{\longleftarrow} R_A) \big)
   :=
   \frac{\im( -, \gamma_A )}{ \im( -, \rho_A )}.
  \]
  \end{construction}
  
For the action of $M$ on morphisms, we need the following lemma.  
  
\begin{lemma}\label{lemma:syz_and_ker}
 A morphism $\alpha: A \rightarrow B$ in $\PC$ induces a well-defined morphism
 between two objects $(A \stackrel{\gamma_A}{\longrightarrow}\Omega_A \stackrel{\rho_A}{\longleftarrow} R_A)$
 and 
 $(B \stackrel{\gamma_B}{\longrightarrow}\Omega_B \stackrel{\rho_B}{\longleftarrow} R_B)$
 in $\ACIC( \PC )$
 if and only if $(-,\alpha)$ restricts as follows:
 \begin{center}
            \begin{tikzpicture}[label/.style={postaction={
              decorate,
              decoration={markings, mark=at position .5 with \node #1;}}},
              mylabel/.style={thick, draw=none, align=center, minimum width=0.5cm, minimum height=0.5cm,fill=white}]
                \coordinate (r) at (3,0);
                \coordinate (u) at (0,2);
                \node (A) {$0$};
                \node (B) at ($(A)+(r)$) {$\ker( \epsilon_A )$};
                \node (C) at ($(B)+(r)$) {$(-,A)$};

                \node (A2) at ($(A)-(u)$) {$0$};
                \node (B2) at ($(A2)+(r)$) {$\ker( \epsilon_B )$};
                \node (C2) at ($(B2)+(r)$) {$(-,B)$.};

                \draw[->,thick] (A) -- (B);
                \draw[->,thick] (B) -- (C);
                
                \draw[->,thick] (A2) -- (B2);
                \draw[->,thick] (B2) -- (C2);
                
                \draw[->,thick,dashed] (B) -- (B2);
                \draw[->,thick] (C) --node[left]{$(-,\alpha)$} (C2);
            \end{tikzpicture}
  \end{center}
\end{lemma}
\begin{proof}
 We compute the evaluation of $\kernel( \epsilon_A )$ at $P \in \PC$:
 \begin{align*}
  \kernel( \epsilon_A )( P ) &= \big\{ (P \stackrel{\sigma}{\rightarrow} A) \mid \epsilon_A( \sigma ) = 0 \big\} \\
  &= \big\{ (P \stackrel{\sigma}{\rightarrow} A) \mid (P, \gamma_A)( \sigma ) \in \im( (P, \rho_A ) ) \big\} \\
  &= \big\{ (P \stackrel{\sigma}{\rightarrow} A) \mid \exists (P \stackrel{\omega}{\rightarrow} R_A): \sigma \cdot \gamma_A = \omega \cdot \rho_A \big\} \\
  &= \big\{ (P \stackrel{\sigma}{\rightarrow} A) \in \SyzC(A \stackrel{\gamma_A}{\longrightarrow}\Omega_A \stackrel{\rho_A}{\longleftarrow} R_A)  \big\}.
 \end{align*}
 Thus, we get a commutative diagram for every evaluation at $P \in \PC$
 \begin{center}
            \begin{tikzpicture}[label/.style={postaction={
              decorate,
              decoration={markings, mark=at position .5 with \node #1;}}},
              mylabel/.style={thick, draw=none, align=center, minimum width=0.5cm, minimum height=0.5cm,fill=white}]
                \coordinate (r) at (2.5,0);
                \coordinate (u) at (0,2);
                \node (A) {$\big\{
                (P \stackrel{\sigma}{\rightarrow} A) \in \SyzC(A \stackrel{\gamma_A}{\longrightarrow}\Omega_A \stackrel{\rho_A}{\longleftarrow} R_A) 
                \big\}$};
                \node (B) at ($(A)+2*(r)$) {$\ker( \epsilon_A )(P)$};
                \node (C) at ($(B)+1.3*(r)$) {$(P,A)$};
                
                \node (A2) at ($(A) - (u)$) {$\big\{
                (P \stackrel{\sigma}{\rightarrow} B) \in \SyzC(B \stackrel{\gamma_B}{\longrightarrow}\Omega_B \stackrel{\rho_B}{\longleftarrow} R_B) 
                \big\}$};
                \node (B2) at ($(A2)+2*(r)$) {$\ker( \epsilon_B )(P)$};
                \node (C2) at ($(B2)+1.3*(r)$) {$(P,B)$};
                \draw[draw=none] (A) --node[]{$=$} (B);
                \draw[right hook->,thick] (B) -- (C);
                \draw[draw=none] (A2) --node[]{$=$} (B2);
                \draw[right hook->,thick] (B2) -- (C2);
                \draw[->,thick] (C) --node[right]{$(P,\alpha)$} (C2);
                \draw[->,thick,dashed] (A) -- (A2);
            \end{tikzpicture}
  \end{center}
 if and only if the well-definedness property holds.
\end{proof}

\begin{construction}\label{construction:functor_aux_on_mor}
  By Lemma \ref{lemma:syz_and_ker} we can define the action of $M$ on a morphism
  \[
   \alpha:
   (A \stackrel{\gamma_A}{\longrightarrow}\Omega_A \stackrel{\rho_A}{\longleftarrow} R_A)
   \longrightarrow
   (B \stackrel{\gamma_B}{\longrightarrow}\Omega_B \stackrel{\rho_B}{\longleftarrow} R_B)
  \]
  in $\ACIC( \PC )$
  by the unique morphism
  completing the following commutative diagram:
  \begin{equation}\label{equation:functor_aux_on_mor}
            \begin{tikzpicture}[label/.style={postaction={
              decorate,
              decoration={markings, mark=at position .5 with \node #1;}}},baseline=(base),
              mylabel/.style={thick, draw=none, align=center, minimum width=0.5cm, minimum height=0.5cm,fill=white}]
                \coordinate (r) at (3,0);
                \coordinate (u) at (0,2);
                \node (A) {$0$};
                \node (B) at ($(A)+(r)$) {$\ker( \epsilon_A )$};
                \node (C) at ($(B)+(r)$) {$(-,A)$};
                \node (D) at ($(C)+(r)$) {$\frac{\im( -, \gamma_A )}{ \im( -, \rho_A )}$};
                \node (E) at ($(D)+(r)$) {$0$};
                
                \node (A2) at ($(A)-(u)$) {$0$};
                \node (B2) at ($(A2)+(r)$) {$\ker( \epsilon_B )$};
                \node (C2) at ($(B2)+(r)$) {$(-,B)$};
                \node (D2) at ($(C2)+(r)$) {$\frac{\im( -, \gamma_B )}{ \im( -, \rho_B )}$};
                \node (E2) at ($(D2)+(r)$) {$0$.};
                
                \node (base) at ($(A)-0.5*(u)$) {};
                
                \draw[->,thick] (A) -- (B);
                \draw[->,thick] (B) -- (C);
                \draw[->,thick] (C) -- node[above]{$\epsilon_A$} (D);
                \draw[->,thick] (D) -- (E);
                
                \draw[->,thick] (A2) -- (B2);
                \draw[->,thick] (B2) -- (C2);
                \draw[->,thick] (C2) -- node[above]{$\epsilon_B$} (D2);
                \draw[->,thick] (D2) -- (E2);
                
                \draw[->,thick] (B) -- (B2);
                \draw[->,thick] (C) --node[left]{$(-,\alpha)$} (C2);
                \draw[->,thick,dashed] (D) --node[right]{$M(\alpha)$} (D2);
            \end{tikzpicture}
  \end{equation}
  Functoriality of $M$ is implied by the functoriality of taking cokernels of commutative squares.
  The same holds for additivity.
\end{construction}

\begin{lemma}\label{lemma:aux_induces_fsubquot}
  Given a morphism
  $
    \alpha:
    (A \stackrel{\gamma_A}{\longrightarrow}\Omega_A \stackrel{\rho_A}{\longleftarrow} R_A)
    \longrightarrow
    (B \stackrel{\gamma_B}{\longrightarrow}\Omega_B \stackrel{\rho_B}{\longleftarrow} R_B)
  $
  in $\ACIC( \PC )$, then
  \[
   M( \alpha ) = 0 \hspace{1em}\Longleftrightarrow\hspace{1em} \alpha \in \mathbf{I}(\PC),
  \]
  where $M$ is the functor described in the Constructions \ref{construction:functor_aux_mod} and \ref{construction:functor_aux_on_mor}
  and $\mathbf{I}(\PC)$ is the ideal defined in Theorem and Definition \ref{theorem:ideal_in_aux}.
\end{lemma}
\begin{proof}
 From the diagram \eqref{equation:functor_aux_on_mor},
 we see that
 $M(\alpha) = 0$ if and only if there exists
 a commutative diagram
 \begin{center}
            \begin{tikzpicture}[label/.style={postaction={
              decorate,
              decoration={markings, mark=at position .5 with \node #1;}}},
              mylabel/.style={thick, draw=none, align=center, minimum width=0.5cm, minimum height=0.5cm,fill=white}]
                \coordinate (r) at (3,0);
                \coordinate (u) at (0,2);
                \node (A) {};
                \node (B) at ($(A)+(r)$) {};
                \node (C) at ($(B)+(r)$) {$(-,A)$};
                \node (D) at ($(C)+(r)$) {};
                \node (E) at ($(D)+(r)$) {};
                
                \node (A2) at ($(A)-(u)$) {};
                \node (B2) at ($(A2)+(r)$) {$\ker( \epsilon_B )$};
                \node (C2) at ($(B2)+(r)$) {$(-,B)$.};
                \node (D2) at ($(C2)+(r)$) {};
                \node (E2) at ($(D2)+(r)$) {};
                
                \draw[->,thick] (B2) -- (C2);
                \draw[->,thick] (C) --node[right]{$(-,\alpha)$} (C2);
                \draw[->,thick,dashed] (C) -- (B2);
            \end{tikzpicture}
  \end{center}
  By the Yoneda lemma, this is equivalent to
  $\alpha \in \kernel( \epsilon_B )(A)$, i.e.,
  $\alpha \in \SyzC(B \stackrel{\gamma_B}{\longrightarrow}\Omega_B \stackrel{\rho_B}{\longleftarrow} R_B)$,
  which exactly means $\alpha \in \mathbf{I}(\PC)$.
\end{proof}

\begin{theorem}\label{theorem:main_theorem}
 The functor $M$ from Construction \ref{construction:functor_aux_mod}
 induces a full and faithful functor
 \[
  \overline{M}: \CIC( \PC ) \rightarrow \Modr \PC
 \]
 that preserves cokernels and images.
\end{theorem}
\begin{proof}
 We use the notation of Construction \ref{construction:functor_aux_mod}.
 Since $\CIC( \PC ) = \ACIC( \PC ) / \mathbf{I}(\PC)$,
 we get a faithful induced additive functor $\overline{M}$ by Lemma \ref{lemma:aux_induces_fsubquot}.
 Furthermore, 
 since representable functors are projectives in $\Modr \PC$,
 every natural transformation
 $\frac{\im( -, \gamma_A )}{ \im( -, \rho_A )} \rightarrow \frac{\im( -, \gamma_B )}{ \im( -, \rho_B )}$
 can be lifted to a natural transformation $(-,A) \rightarrow (-,B)$
 and from Lemma \ref{lemma:syz_and_ker}, it follows that $\overline{M}$ is full.
 
 Next, let
 \[
 \overline{\alpha}: (A \stackrel{\gamma_A}{\longrightarrow}\Omega_A \stackrel{\rho_A}{\longleftarrow} R_A) \longrightarrow (B \stackrel{\gamma_B}{\longrightarrow}\Omega_B \stackrel{\rho_B}{\longleftarrow} R_B)
 \]
 denote an arbitrary morphism in $\CIC( \PC )$.
 We have a commutative diagram of the form
 \begin{center}
 \begin{tikzpicture}[label/.style={postaction={
              decorate,
              decoration={markings, mark=at position .5 with \node #1;}}}
              mylabel/.style={thick, draw=none, align=center, minimum width=0.5cm, minimum height=0.5cm,fill=white}]
                \coordinate (r) at (3,0);
                \coordinate (u) at (0,2);
               
                \node (C)  {$(-,A)$};
                \node (D) at ($(C)+(r)$) {$\frac{\im( -, \gamma_A )}{ \im( -, \rho_A )}$};

                \node (C2) at ($(C)-(u)$) {$(-,B)$};
                \node (D2) at ($(C2)+(r)$) {$\frac{\im( -, \gamma_B )}{ \im( -, \rho_B )}.$};

                \draw[->>,thick] (C) -- node[above]{$\epsilon_A$} (D);
                
                \draw[->>,thick] (C2) -- node[above]{$\epsilon_B$} (D2);

                \draw[->,thick] (C) --node[left]{$(-,\alpha)$} (C2);
                \draw[->,thick] (D) --node[right]{$\overline{M}(\overline{\alpha})$} (D2);
 \end{tikzpicture}
 \end{center}
 We compute
 \begin{align*}
  \image\left( \overline{M}(\overline{\alpha}) \right) &= \image\left( \epsilon_A \cdot \overline{M}(\overline{\alpha}) \right) \\
  &= \image\left( (-,\alpha) \cdot \epsilon_B \right) \\
 \end{align*}
 which yields for every $P \in \PC$:
 \begin{align*}
  \image\left( \overline{M}(\overline{\alpha}) \right)(P) 
  &=
  \left\{ (P \stackrel{\sigma \cdot ( \alpha \cdot \gamma_B )}{\longrightarrow} \Omega_B) + \left(\image( P, \rho_B) \cap \image( P, \gamma_B ) \right) \mid \sigma \in (P,A) \right\} \\
  &=
  \left\{ (P \stackrel{\iota}{\longrightarrow} \Omega_B) + \left(\image( P, \rho_B) \cap \image( P, \gamma_B ) \right) \mid \iota \in \image(P,\alpha \cdot \gamma_B) \right\}. \\
 \end{align*}
 Thus, we can describe the cokernel projection of $\overline{M}(\overline{\alpha})$
 by the right vertical morphism in the diagram
\begin{center}
 \begin{tikzpicture}[label/.style={postaction={
              decorate,
              decoration={markings, mark=at position .5 with \node #1;}}}
              mylabel/.style={thick, draw=none, align=center, minimum width=0.5cm, minimum height=0.5cm,fill=white}]
                \coordinate (r) at (6,0);
                \coordinate (u) at (0,2);
               
                \node (C)  {$(-,B)$};
                \node (D) at ($(C)+(r)$) {$\frac{\im( -, \gamma_B )}{ \im( -, \rho_B )}$};

                \node (C2) at ($(C)-(u)$) {$(-,B)$};
                \node (D2) at ($(C2)+(r)$) {$\frac{\im( -, \gamma_B )}{ \im( -, \rho_B ) + \im( -,\alpha \cdot \gamma_B)}$,};

                \draw[->>,thick] (C) -- node[above]{$\epsilon_B$} (D);
                
                \draw[->>,thick] (C2) -- node[above]{$\epsilon_B$} (D2);

                \draw[->,thick] (C) --node[left]{$(-,\id_B)$} (C2);
                \draw[->,thick] (D) --node[right]{} (D2);
 \end{tikzpicture}
 \end{center}
 which is exactly the application of $\overline{M}$ to the cokernel projection described in Construction \ref{construction:cokernel}.
 Thus, $\overline{M}$ respects cokernels.
 
 To show that $\overline{M}$ respects images, it suffices to prove that it respects monos and epis, since this
 implies that it respects epi-mono factorizations.
 Since $\overline{M}$ is additive and respects cokernels, it follows that $\overline{M}$ respects epimorphisms.
 Now, let
 \[
 \overline{\alpha}: (A \stackrel{\gamma_A}{\longrightarrow}\Omega_A \stackrel{\rho_A}{\longleftarrow} R_A) \longrightarrow (B \stackrel{\gamma_B}{\longrightarrow}\Omega_B \stackrel{\rho_B}{\longleftarrow} R_B)
 \]
 denote a mono in $\CIC( \PC )$.
 In order to test whether $\overline{M}(\overline{\alpha})$ is a mono,
 the Yoneda lemma implies that it suffices to check test morphisms of the form
 \[
  \tau: (-,P) \longrightarrow \overline{M}(A \stackrel{\gamma_A}{\longrightarrow}\Omega_A \stackrel{\rho_A}{\longleftarrow} R_A).
 \]
 So, given $\tau$ as above which also is a test morphism, i.e., such that $\tau \cdot \overline{M}(\overline{\alpha}) = 0$,
 it can be written as
\[
  \overline{M}( \overline{\tau'} ): \overline{M}(\emb(P)) \longrightarrow \overline{M}(A \stackrel{\gamma_A}{\longrightarrow}\Omega_A \stackrel{\rho_A}{\longleftarrow} R_A)
\]
for a uniquely determined 
\[
 \overline{\tau'}: \emb(P) \rightarrow (A \stackrel{\gamma_A}{\longrightarrow}\Omega_A \stackrel{\rho_A}{\longleftarrow} R_A)
\]
in $\CIC( \PC )$ (in fact, $\overline{M} \circ \emb$ is the Yoneda embedding). Since $\overline{\alpha}$ is a mono,
it follows that $\overline{\tau'} = 0$, and thus $\overline{M}( \overline{\tau'} ) = \tau = 0$.
\end{proof}

Recall that a subcategory $\AC$ of a category $\BC$ is called \textbf{replete}
if for any $X \in \AC$ and isomorphism $\iota: X \rightarrow Y$ in $\BC$,
$\iota$ belongs to $\AC$.
We get a characterization of the \textbf{essential image} $\im( \overline{M} ) \subseteq \Modr\PC$,
i.e., the smallest full replete subcategory generated by all objects
of the form $\overline{M}(A \stackrel{\gamma_A}{\longrightarrow}\Omega_A \stackrel{\rho_A}{\longleftarrow} R_A)$.
Note that by Theorem \ref{theorem:main_theorem},
we have an equivalence $\CIC( \PC ) \simeq \im( \overline{M} )$.

\begin{corollary}\label{corollary:characterization_as_cokernel_image_closure}
 The essential image of $\overline{M}$ is given by the smallest full and replete additive subcategory $\FC\subseteq \Modr\PC$
 with the following properties:
 \begin{enumerate}
  \item $\PC \subseteq \FC$ via the Yoneda embedding,
  \item $\FC$ is closed under taking cokernels in $\Modr\PC$,
  \item $\FC$ is closed under taking images in $\Modr\PC$.
 \end{enumerate}
\end{corollary}
\begin{proof}
 The essential image of $\overline{M}$ satisfies these three properties by Theorem \ref{theorem:main_theorem}.
 Conversely, every $\FC$ satisfying these properties has
 to contain the subquotients
 \begin{equation}\label{equation:im_quo_as_im_of_fp}
 \frac{\im(-,{\gamma_A})}{\im(-,{\rho_A})} \simeq \im\big( (-,A) \stackrel{(-,\gamma_A)}{\longrightarrow} (-,\Omega_A) \longrightarrow \cokernel(-,\rho_A) \big)
 \end{equation}
 for a given cospan
 $(A \stackrel{\gamma_A}{\longrightarrow}\Omega_A \stackrel{\rho_A}{\longleftarrow} R_A)$
 in $\PC$, and thus has to contain the essential image of $\overline{M}$.
\end{proof}

We give a short interlude on some well-known facts about the category of finitely presented functors $\fp( \PC^{\op}, \Ab )$.
For an abstract treatment of $\fp( \PC^{\op}, \Ab )$, see \cite{FreydRep} or \cite{BelFredCats},
for a constructive treatment, see \cite{PosFreyd}.

An additive functor $F: \PC^{\op} \rightarrow \Ab$
is called \textbf{finitely presented} if there exists an exact sequence
\begin{center}
            \begin{tikzpicture}[label/.style={postaction={
              decorate,
              decoration={markings, mark=at position .5 with \node #1;}}},
              mylabel/.style={thick, draw=none, align=center, minimum width=0.5cm, minimum height=0.5cm,fill=white}]
                \coordinate (r) at (3,0);
                \coordinate (u) at (0,1.2);
                \node (B) at ($(A)$) {$(-,B)$};
                \node (C) at ($(B)+(r)$) {$(-,A)$};
                \node (D) at ($(C)+(r)$) {$F$};
                \node (E) at ($(D)+(r)$) {$0$};
                \draw[->,thick] (B) --node[above]{$(-,\alpha)$} (C);
                \draw[->,thick] (C) -- (D);
                \draw[->,thick] (D) -- (E);
            \end{tikzpicture}
\end{center}
in $\Modr \PC$ for a morphism $\alpha: B \rightarrow A$ in $\PC$.
Now, $\fp(\PC^{\op}, \Ab)$ is defined as the full subcategory of $\Modr\PC$ generated by all finitely presented functors.
The additive category $\fp(\PC^{\op}, \Ab)$ is closed under taking cokernels in $\Modr\PC$,
and thus can be characterized similarly to $\im( \overline{M} )$:
it is the smallest full and replete additive subcategory of $\Modr\PC$ which contains all representable functors and
is closed under taking cokernels.
In particular, Corollary \ref{corollary:characterization_as_cokernel_image_closure}
implies $\fp(\PC^{\op}, \Ab) \subseteq \im( \overline{M} )$, which brings us to a second characterization.

\begin{corollary}\label{corollary:characterization_as_cokernel_image_closure_2}
  The essential image of $\overline{M}$ is given by the smallest full and replete additive subcategory $\FC\subseteq \Modr\PC$
  with the following properties:
  \begin{enumerate}
   \item $\fp(\PC^{\op}, \Ab) \subseteq \FC$,
   \item $\FC$ is closed under taking images in $\Modr\PC$.
  \end{enumerate}
 \end{corollary}
\begin{proof}
  Equation \eqref{equation:im_quo_as_im_of_fp} in the proof of Corollary \ref{corollary:characterization_as_cokernel_image_closure}
  shows that every object in $\im( \overline{M} )$ is given as an image of a morphism between finitely presented functors.
\end{proof}
 
As we have seen in Section \ref{section:subquotients},
the category $\im( \overline{M} )$ admits a diagrammatic approach via the category $\CIC( \PC )$
which allows for a computer implementation.
The same is true for the category $\fp(\PC^{\op}, \Ab)$: it is equivalent to the so-called Freyd category
$\Freyd( \PC )$ whose objects are given by morphisms $(A\stackrel{\rho}{\leftarrow} R)$ in $\PC$,
and a morphism from $(A \stackrel{\rho}{\leftarrow} R)$ to $(A' \stackrel{\rho'}{\leftarrow} R')$
is given by a morphism $A \stackrel{\alpha}{\rightarrow} A'$ such that there exists
a morphism $R\stackrel{\omega}{\rightarrow} R'$ which renders the diagram
\begin{center}
  \begin{tikzpicture}[label/.style={postaction={
    decorate,
    decoration={markings, mark=at position .5 with \node #1;}}},
    mylabel/.style={thick, draw=none, align=center, minimum width=0.5cm, minimum height=0.5cm,fill=white}]
      \coordinate (r) at (3,0);
      \coordinate (d) at (0,-1.5);
      \node (B) {$R$};
      \node (C) at ($(B)-(r)$) {$A$};
      \node (D) at ($(A)+(d)$) {$R'$};
      \node (E) at ($(D)-(r)$) {$A'$};
      \draw[->,thick] (B) --node[above]{$\rho$} (C);
      \draw[->,dashed,thick] (B) --node[right]{$\omega$} (D);
      \draw[->,thick] (C) --node[left]{$\alpha$} (E);
      \draw[->,thick] (D) --node[above]{$\rho'$} (E);
  \end{tikzpicture}
\end{center}
commutative. Such a diagram represents the zero morphism if and only if $\alpha$ factors as
$\alpha = \lambda \cdot \rho'$ for some morphism $A \stackrel{\lambda}{\rightarrow} R'$.
For details on possible constructions in $\Freyd( \PC )$, we refer the reader to \cite[Section 3]{PosFreyd}.
We have an equivalence of categories given by:
\[
  \Freyd( \PC ) \rightarrow \fp(\PC^{\op}, \Ab): (A \stackrel{\rho}{\leftarrow} R) \mapsto \cokernel( (-,\rho) ).
\]
Moreover, it is easy to see that the mapping
\[
  \Freyd( \PC ) \rightarrow \CIC( \PC ): (A \stackrel{\rho}{\leftarrow} R) \mapsto (A \stackrel{\id_A}{\longrightarrow} A \stackrel{\rho}{\longleftarrow} R)
\]
gives rise to a functor such that we end up with a diagram of functors
\begin{center}
  \begin{tikzpicture}[label/.style={postaction={
    decorate,
    decoration={markings, mark=at position .5 with \node #1;}}},
    mylabel/.style={thick, draw=none, align=center, minimum width=0.5cm, minimum height=0.5cm,fill=white}]
      \coordinate (r) at (5,0);
      \coordinate (d) at (0,-1.5);
      \node (Fp) {$\fp(\PC^{\op}, \Ab)$};
      \node (I) at ($(B)+(r)$) {$\im( \overline{M} )$};
      \node (R) at ($(I)+0.5*(r)$) {$\Modr\PC$};
      \node (A) at ($(A)+(d)$) {$\Freyd( \PC )$};
      \node (Q) at ($(D)+(r)$) {$\CIC( \PC )$};
      \draw[right hook->,thick] (Fp) -- (I);
      \draw[right hook->,thick] (A) -- (Q);
      \draw[right hook->,thick] (I) -- (R);
      \draw[->,thick] (A) --node[rotate=90,yshift=0.2em]{$\sim$} (Fp);
      \draw[->,thick] (Q) --node[rotate=90,yshift=0.2em]{$\sim$} (I);
  \end{tikzpicture}
\end{center}
commutative up to natural isomorphism.
In the next section, we will characterize the case in which the inclusion
$\fp(\PC^{\op}, \Ab) \subseteq \im( \overline{M} )$ is an equivalence.

\section{The abelian case}\label{section:abelian_case}

The goal of this section is to prove 
the following characterization of the abelian case.

\begin{theorem}\label{theorem:abelian_case}
 The following are equivalent:
 \begin{enumerate}
  \item $\PC$ has weak kernels,
  \item $\CIC( \PC )$ has kernels,
  \item $\CIC( \PC )$ is abelian,
  \item $\fp( \PC^{\op}, \Ab )$ is abelian,
  \item $\CIC( \PC )$ and $\fp( \PC^{\op}, \Ab )$ are equal as full and replete subcategories of $\Modr \PC$,
  \item $\CIC( \PC )$ and $\fp( \PC^{\op}, \Ab )$ are equivalent as (abstract) categories,
  \item $\fp( \PC^{\op}, \Ab )$ has epi-mono factorizations.
 \end{enumerate}
\end{theorem}

The first two subsections in this section are devoted to the construction 
of kernels in $\CIC( \PC )$, and the third subsection to
the proof of Theorem \ref{theorem:abelian_case}.

\subsection{A weakening of weak pullbacks}
A \textbf{weak limit} of a diagram in a category
can be defined exactly as one would define a limit, but without
requiring the morphism induced by its universal property
to be uniquely determined.
Applied to the concept of a pullback, the resulting notion is known as a \textbf{weak pullback}.
In this section, we introduce a further weakening:
we give up the commutativity of one of the two resulting 
triangles in the common pullback diagram describing its
universal property.

\begin{definition}\label{definition:biased_weak_pullbacks}
 Let $\PC$ be an additive category. For a given cospan $A \stackrel{\alpha}{\longrightarrow} B \stackrel{\gamma}{\longleftarrow} C$ in $\PC$,
 a \textbf{biased weak pullback} consists of the following data:
 \begin{enumerate}
  \item An object $P( \alpha, \gamma ) \in \PC$.
  \item A morphism $\pi( \alpha, \gamma ): P( \alpha, \gamma ) \rightarrow A$
        with the property that there exists another morphism $\omega: P( \alpha, \gamma ) \rightarrow C$
        with $\omega \cdot \gamma = \pi( \alpha, \gamma ) \cdot \alpha$.
        We call $\pi( \alpha, \gamma )$ the \textbf{biased weak pullback projection}.
  \item An \operation that \constructs for $T \in \PC$ and a morphism
        $\tau: T \rightarrow A$ with the property $\exists \sigma: T \rightarrow C: \tau \cdot \alpha = \sigma \cdot \gamma$
        a morphism $u(\tau): T \rightarrow P(\alpha, \gamma)$ satisfying
        \begin{center}
        $\tau = u(\tau) \cdot \pi( \alpha, \gamma )$.
        \end{center}
 \end{enumerate}
 Thus, we have the following diagram in which only the indicated parts commute,
 and the dashed morphism is not necessarily uniquely determined:
 
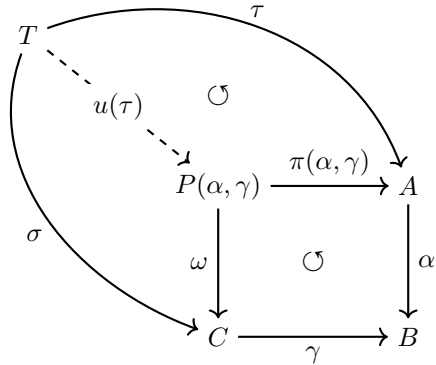
\begin{center}
          \begin{tikzpicture}[label/.style={postaction={
            decorate,
            decoration={markings, mark=at position .5 with \node #1;}}}, baseline = (D),
            mylabel/.style={thick, draw=none, align=center, minimum width=0.5cm, minimum height=0.5cm,fill=white}]
              \coordinate (r) at (2.5,0);
              \coordinate (u) at (0,2);
              \node (A) {$C$};
              \node (B) at ($(A)+1*(r)$) {$B$};
              \node (C) at ($(B) +(u)$) {$A$};
              \node (D) at ($(A) + (u)$) {$P(\alpha, \gamma)$};
              \node (T) at ($(D) + (u) - (2.5,0)$) {$T$};
              \node (comm1) at ($(A) + 0.5*(r) + 0.5*(u)$) {$\circlearrowleft$};
              \node (comm2) at ($(T) + 1*(r) - 0.4*(u)$) {$\circlearrowleft$};
              \draw[->,thick] (A) -- node[below]{$\gamma$} (B);
              \draw[->,thick] (C) -- node[right]{$\alpha$} (B);
              \draw[->,thick] (D) -- node[left]{$\omega$} (A);
              \draw[->,thick] (D) -- node[above]{$\pi( \alpha, \gamma )$} (C);
              \draw (T) [->,thick,out=250,in=160] to node[below,xshift=-0.5em]{$\sigma$} (A);
              \draw (T) [->,thick,out=20,in=110] to node[above]{$\tau$} (C);
              \draw[->,thick,dashed] (T) -- node[mylabel]{$u(\tau)$} (D);
          \end{tikzpicture}
        \end{center}
\end{definition}

\begin{remark}
 A biased weak pullback of a given cospan $A \stackrel{\alpha}{\longrightarrow} B \stackrel{\gamma}{\longleftarrow} C$
 is the same as a weak terminal object in $\SyzC(A \stackrel{\alpha}{\longrightarrow} B \stackrel{\gamma}{\longleftarrow} C)$.
\end{remark}

We say that $\PC$ \textbf{has biased weak pullbacks}
if it comes equipped with an operation constructing the triple $(P( \alpha, \gamma ), \pi( \alpha, \gamma ), u )$
for given input cospan $A \stackrel{\alpha}{\longrightarrow} B \stackrel{\gamma}{\longleftarrow} C$. 

\begin{lemma}\label{lemma:biased_weak_pb_weak_kernels}
 The following are equivalent:
 \begin{enumerate}
  \item $\PC$ has biased weak pullbacks,
  \item $\PC$ has weak pullbacks,
  \item $\PC$ has weak kernels.
 \end{enumerate}
\end{lemma}
\begin{proof}
 If $\PC$ has biased weak pullbacks, 
 then $P(A \stackrel{{\alpha}}{\longrightarrow} B, 0 {\longrightarrow} B)$ is a weak kernel of $\alpha$.
 Moreover, we can construct weak pullbacks from direct sums and weak kernels.
 Last, every weak pullback is also a biased weak pullback.
\end{proof}

Despite the statement of Lemma \ref{lemma:biased_weak_pb_weak_kernels},
biased weak pullbacks are important for us because of two reasons:
\begin{enumerate}
 \item They have fewer constraints than weak pullbacks and are thus easier to compute.
 \item They are all we need in the construction of kernels in $\CIC(\PC)$.
\end{enumerate}
We demonstrate the first of these arguments
with our running example $\Rows_R$.

\begin{lemma}\label{lemma:biased_weak_pullbacks_rows}
 Let $a,b,c,p \in \Nzero$. A commutative square in $\Rows_R$
 \begin{center}
          \begin{tikzpicture}[label/.style={postaction={
            decorate,
            decoration={markings, mark=at position .5 with \node #1;}}}, baseline = (D),
            mylabel/.style={thick, draw=none, align=center, minimum width=0.5cm, minimum height=0.5cm,fill=white}]
              \coordinate (r) at (2.5,0);
              \coordinate (u) at (0,2);
              \node (A) {$R^{1 \times c}$};
              \node (B) at ($(A)+1*(r)$) {$R^{1 \times b}$};
              \node (C) at ($(B) +(u)$) {$R^{1 \times a}$};
              \node (D) at ($(A) + (u)$) {$R^{1 \times p}$};
              \node (comm1) at ($(A) + 0.5*(r) + 0.5*(u)$) {$\circlearrowleft$};
              \draw[->,thick] (A) -- node[below]{$\gamma$} (B);
              \draw[->,thick] (C) -- node[right]{$\alpha$} (B);
              \draw[->,thick] (D) -- node[left]{$\omega$} (A);
              \draw[->,thick] (D) -- node[above]{$\pi$} (C);
          \end{tikzpicture}
        \end{center}
 is a biased weak pullback in $\Rows_R$ with biased weak pullback projection $\pi$ if and only if
 \[
  \im( \pi ) = \alpha^{-1}( \im( \gamma ) )
 \]
 as submodules of $R^{1 \times a }$ in $R\Modl$.
\end{lemma}
\begin{proof}
Whenever we have a commutative square of the form
\begin{center}
    \begin{tikzpicture}[label/.style={postaction={
      decorate,
      decoration={markings, mark=at position .5 with \node #1;}}}, baseline = (D),
      mylabel/.style={thick, draw=none, align=center, minimum width=0.5cm, minimum height=0.5cm,fill=white}]
        \coordinate (r) at (2.5,0);
        \coordinate (u) at (0,2);
        \node (A) {$R^{1 \times c}$};
        \node (B) at ($(A)+1*(r)$) {$R^{1 \times b}$,};
        \node (C) at ($(B) +(u)$) {$R^{1 \times a}$};
        \node (D) at ($(A) + (u)$) {$R^{1 \times t}$};
        \node (comm1) at ($(A) + 0.5*(r) + 0.5*(u)$) {$\circlearrowleft$};
        \draw[->,thick] (A) -- node[below]{$\gamma$} (B);
        \draw[->,thick] (C) -- node[right]{$\alpha$} (B);
        \draw[->,thick] (D) -- node[left]{$\sigma$} (A);
        \draw[->,thick] (D) -- node[above]{$\tau$} (C);
    \end{tikzpicture}
\end{center}
we have an inclusion $\im( \tau ) \subseteq \alpha^{-1}( \im( \gamma ) )$.
Now, if $\im( \pi ) = \alpha^{-1}( \im( \gamma ) )$,
then we get a morphism $u(\tau)$ by the projectivity of $R^{1 \times t}$ in $R\Modl$
rendering the diagram
\begin{center}
    \begin{tikzpicture}[label/.style={postaction={
      decorate,
      decoration={markings, mark=at position .5 with \node #1;}}}, baseline = (D),
      mylabel/.style={thick, draw=none, align=center, minimum width=0.5cm, minimum height=0.5cm,fill=white}]
        \coordinate (r) at (2.5,0);
        \coordinate (u) at (0,1.5);
        \node (A) {$R^{1 \times p}$};
        \node (B) at ($(A)+2*(r)$) {$\im( \pi )$};
        \node (C) at ($(B)+(r)$) {$R^{1 \times a}$};
        \node (D) at ($(B) + (u)$) {$\im(\tau)$};
        \node (E) at ($(D) + (u)$) {$R^{1 \times t}$};
        \node (comm1) at ($(A) + 1.25*(r) + 0.65*(u)$) {$\circlearrowleft$};
        \draw[->>,thick] (A) -- (B);
        \draw[right hook->,thick] (D) -- (B);
        \draw[right hook->,thick] (B) -- (C);
        \draw[->,thick] (E) -- (D);
        \draw[->,thick,dashed] (E) --node[left,yshift=0.5em]{$u(\tau)$}  (A);
    \end{tikzpicture}
\end{center}
commutative. Thus, we get a biased weak pullback.
Conversely, let $R^{1 \times p}$ and $\pi$ define a biased weak pullback.
For a given $v \in \alpha^{-1}( \im( \gamma ) )$ there is a $w \in R^{1 \times c}$
such that we get a commutative diagram
\begin{center}
    \begin{tikzpicture}[label/.style={postaction={
      decorate,
      decoration={markings, mark=at position .5 with \node #1;}}}, baseline = (D),
      mylabel/.style={thick, draw=none, align=center, minimum width=0.5cm, minimum height=0.5cm,fill=white}]
        \coordinate (r) at (2.5,0);
        \coordinate (u) at (0,2);
        \node (A) {$R^{1 \times c}$};
        \node (B) at ($(A)+1*(r)$) {$R^{1 \times b}$};
        \node (C) at ($(B) +(u)$) {$R^{1 \times a}$};
        \node (D) at ($(A) + (u)$) {$R^{1 \times 1}$};
        \node (comm1) at ($(A) + 0.5*(r) + 0.5*(u)$) {$\circlearrowleft$};
        \draw[->,thick] (A) -- node[below]{$\gamma$} (B);
        \draw[->,thick] (C) -- node[right]{$\alpha$} (B);
        \draw[->,thick] (D) -- node[left]{$w$} (A);
        \draw[->,thick] (D) -- node[above]{$v$} (C);
    \end{tikzpicture}
\end{center}
where we identify the element $v$ (resp. $w$) with the map starting from $R^{1 \times 1}$
that sends $1$ to $v$ (resp. $w$).
Using the weak universal property, we get
\[
 v = u(v) \cdot \pi
\]
which means $v \in \im( \pi )$.
\end{proof}

Using Lemma \ref{lemma:biased_weak_pullbacks_rows} we can demonstrate
that a biased weak pullback can significantly differ from a weak pullback.
We provide a simple example:

\begin{example}\label{example:biased_weak_pullbacks_rows}
 By Lemma \ref{lemma:biased_weak_pullbacks_rows},
 the cospan $R^{1 \times a} \stackrel{0}{\longrightarrow} 0 \stackrel{0}{\longleftarrow} R^{1 \times c}$ in $\Rows_R$
 admits a biased weak pullback with projection $R^{1 \times a} \stackrel{\id}{\longrightarrow} R^{1 \times a}$.
 Assume there exists an $\omega: R^{1 \times a} \rightarrow R^{1 \times c}$ such that
 $\id_{R^{1 \times a}}$ and $\omega$ define the projections of a weak pullback.
 Then there has to exist a commutative diagram of the form
 
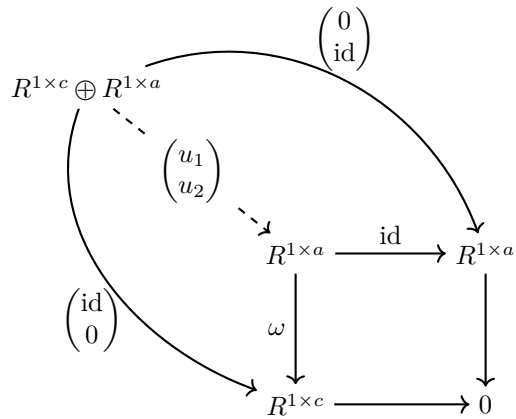
\begin{center}
          \begin{tikzpicture}[label/.style={postaction={
            decorate,
            decoration={markings, mark=at position .5 with \node #1;}}}, baseline = (D),
            mylabel/.style={thick, draw=none, align=center, minimum width=0.5cm, minimum height=0.5cm,fill=white}]
              \coordinate (r) at (2.5,0);
              \coordinate (u) at (0,2);
              \node (A) {$R^{1 \times c}$};
              \node (B) at ($(A)+1*(r)$) {$0$};
              \node (C) at ($(B) +(u)$) {$R^{1 \times a}$};
              \node (D) at ($(A) + (u)$) {$R^{1 \times a}$};
              \node (T) at ($(D) + 1.1*(u) - 1.1*(2.5,0)$) {$R^{1 \times c} \oplus R^{1 \times a}$};
              \draw[->,thick] (A) -- (B);
              \draw[->,thick] (C) -- (B);
              \draw[->,thick] (D) -- node[left]{$\omega$} (A);
              \draw[->,thick] (D) -- node[above]{$\id$} (C);
              \draw (T) [->,thick,out=250,in=160] to node[below,xshift=-0.5em]{$\pmatcol{\id}{0}$} (A);
              \draw (T) [->,thick,out=20,in=110] to node[above]{$\pmatcol{0}{\id}$} (C);
              \draw[->,thick,dashed] (T) -- node[mylabel]{$\pmatcol{u_1}{u_2}$} (D);
          \end{tikzpicture}
        \end{center}
   which is absurd if $c > 0$, since commutativity of the upper triangle implies $u_1 = 0$, $u_2 = \id$,
   and commutativity of the lower triangle implies
   \[
    \id = u_1 \cdot \omega = 0.
   \]
   Note that $R^{1 \times c} \oplus R^{1 \times a}$ together with its projections
  to its factors is actually a (weak) pullback of the given cospan, 
  so, this example demonstrates that the computation of biased weak pullbacks instead of weak
pullbacks might result in a significant 
decrease in the number of needed generators (in this concrete example, we save $c$-many generators).
\end{example}

For computational reasons, whenever it suffices to work with biased weak pullbacks
instead of weak pullbacks, one should do so.

\subsection{Kernels}

We show how to construct
kernels in $\CIC( \PC )$ provided $\PC$ has biased weak pullbacks.

\begin{construction}\label{construction:kernels}
 Given a morphism
 \[\overline{\alpha}: (A \stackrel{\gamma_A}{\longrightarrow}\Omega_A \stackrel{\rho_A}{\longleftarrow} R_A) \longrightarrow (B \stackrel{\gamma_B}{\longrightarrow}\Omega_B \stackrel{\rho_B}{\longleftarrow} R_B)\]
 in $\CIC(\PC)$, the following diagram
 shows us how to construct its kernel embedding along with the universal property:
 \begin{center}
    \begin{tikzpicture}[label/.style={postaction={
          decorate,
          decoration={markings, mark=at position .5 with \node #1;}},
          mylabel/.style={thick, draw=black, align=center, minimum width=0.5cm, minimum height=0.5cm,fill=white}}]
          \coordinate (r) at (5.5,0);
          \coordinate (u) at (0,2);
          \node (A) {$(A \stackrel{\gamma_A}{\longrightarrow}\Omega_A \stackrel{\rho_A}{\longleftarrow} R_A)$};
          \node (B) at ($(A)+0.8*(r)$) {$(B \stackrel{\gamma_B}{\longrightarrow}\Omega_B \stackrel{\rho_B}{\longleftarrow} R_B)$.};
          \node (K) at ($(A) - 1*(r) + (u)$) {$(P(\alpha \cdot \gamma_B, \rho_B) \stackrel{\pi(\alpha \cdot \gamma_B, \rho_B) \cdot \gamma_A}{\longrightarrow} \Omega_A \stackrel{\rho_A}{\longleftarrow} R_A)$};
          \node (T) at ($(A) - 1*(r) - (u)$) {$(T \stackrel{\gamma_T}{\longrightarrow} \Omega_T \stackrel{\rho_T}{\longleftarrow} R_T)$};
          \draw[->,thick] (A) --node[above]{$\overline{\alpha}$} (B);
          \draw[->,thick] (T) --node[below,yshift=-0.2em]{$\overline{\tau}$} (A);
          \draw[->,thick] (K) --node[above,xshift=3.2em,yshift=-0.3em]{$\overline{\pi(\alpha \cdot \gamma_B, \rho_B)}$} (A);
          \draw[->,thick,dashed,label={[mylabel]{$\overline{u(\tau)}$}}] (T) -- (K);
    \end{tikzpicture}
 \end{center}
 How to read this diagram:
 the solid arrow pointing down right is the kernel embedding,
 the solid arrow pointing up right is a test morphism for the universal property of the kernel,
 and the dashed arrow pointing straight up is the morphism induced by the universal property.
 The biased weak pullback diagram needed in this construction looks as follows:
 \begin{center}
          \begin{tikzpicture}[label/.style={postaction={
            decorate,
            decoration={markings, mark=at position .5 with \node #1;}}}, baseline = (D),
            mylabel/.style={thick, draw=none, align=center, minimum width=0.5cm, minimum height=0.5cm,fill=white}]
              \coordinate (r) at (4.5,0);
              \coordinate (u) at (0,2);
              \node (A) {$R_B$};
              \node (B) at ($(A)+1*(r)$) {$\Omega_B$};
              \node (C) at ($(B) +(u)$) {$A$};
              \node (D) at ($(A) + (u)$) {$P(\alpha \cdot \gamma_B, \rho_B)$};
              \node (T) at ($(D) + (u) - (2.5,0)$) {$T$};
              \node (comm1) at ($(A) + 0.5*(r) + 0.5*(u)$) {$\circlearrowleft$};
              \node (comm2) at ($(T) + 1*(r) - 0.4*(u)$) {$\circlearrowleft$};
              \draw[->,thick] (A) -- node[below]{$\rho_B$} (B);
              \draw[->,thick] (C) -- node[right]{$\alpha \cdot \gamma_B$} (B);
              \draw[->,thick] (D) -- node[left]{$\omega$} (A);
              \draw[->,thick] (D) -- node[above]{$\pi( \alpha \cdot \gamma_B, \rho_B )$} (C);
              \draw (T) [->,thick,out=250,in=160] to node[below,xshift=-0.5em]{$\zeta$} (A);
              \draw (T) [->,thick,out=20,in=110] to node[above]{$\tau$} (C);
              \draw[->,thick,dashed] (T) -- node[mylabel]{$u(\tau)$} (D);
          \end{tikzpicture}
  \end{center}
  Note that $\zeta$ is simply a witness for the composition $\overline{\tau}\cdot \overline{\alpha}$ being zero.
\end{construction}
\begin{proof}[Correctness of the construction]
 To shorten notation we 
 denote the candidate for the kernel object by $\widehat{K}$.
 Any syzygy witness of a $\sigma \in \SyzC( \widehat{K} )$
 can also be used
 as a syzygy witness of
 $\sigma \cdot \pi(\alpha \cdot \gamma_B, \rho_B)$ in $\SyzC(A \stackrel{\gamma_A}{\longrightarrow}\Omega_A \stackrel{\rho_A}{\longleftarrow} R_A)$.
 Thus, the well-definedness property of the kernel embedding holds.
 
 Furthermore, we can take $\omega$ as a witness for the composition of the kernel embedding with $\overline{\alpha}$ being zero.
 Moreover, the kernel embedding is a mono
 by Lemma \ref{lemma:decide_mono}.
  
 Next,
 let $\sigma \in \SyzC(T \stackrel{\gamma_T}{\longrightarrow} \Omega_T \stackrel{\rho_T}{\longleftarrow} R_T)$.
 Then $\sigma \cdot \tau \in \SyzC(A \stackrel{\gamma_A}{\longrightarrow}\Omega_A \stackrel{\rho_A}{\longleftarrow} R_A)$,
 and since $\tau = u(\tau) \cdot \pi( \alpha \cdot \gamma_B, \rho_B )$,
 it follows that $\sigma \cdot u(\tau) \in \SyzC( \widehat{K} )$.
 Thus,
 the well-definedness property of the kernel induced morphism holds.
 
 Last, the commutativity of the triangle in the kernel diagram already holds in $\ACIC( \PC )$.
\end{proof}
Note that at no point in this proof did we need commutativity of the lower triangle in
the biased weak pullback diagram. This justifies our introduction of the
concept of a biased weak pullback.

\subsection{Proof of the characterization of the abelian case}\label{subsection:proof}

\begin{proof}[Proof of the equivalence of statements $(1)-(3)$ in Theorem \ref{theorem:abelian_case}] 
 \
 
 $(1) \implies (2)$:
 If $\PC$ has weak kernels,
 then it has biased weak pullbacks by Lemma \ref{lemma:biased_weak_pb_weak_kernels}.
 It follows from Construction \ref{construction:kernels} that $\CIC( \PC )$ has kernels.
 
 $(2) \implies (3)$:
 Construction \ref{construction:colift_along_epi} proves that
 every epimorphism is the cokernel projection of its kernel embedding
 in the case when $\CIC( \PC )$ has kernels, which is true by assumption.
 All the other axioms of an abelian category hold due to the constructions in Section \ref{section:subquotients}.
 
 $(3) \implies (1)$:
 Given a morphism $\alpha: A \rightarrow B$ in $\PC$, 
 compute the kernel embedding
 \[
  \overline{\kappa}: (K \rightarrow \Omega_K \leftarrow R_K) \longrightarrow \emb(A)
 \]
 of $\emb(\alpha)$ in $\CIC( \PC )$.
 Then $\kappa: K \rightarrow A$ is a weak kernel of $\alpha$.
\end{proof}

We say an additive functor $F: \PC^{\op} \rightarrow \Ab$
is \textbf{finitely generated} if it admits an epimorphism
\[
 (-,A) \twoheadrightarrow F
\]
in $\Modr\PC$ for some $A \in \PC$.
We will need the facts listed in the following lemma,
for which we will provide proofs for
the sake of completeness.
See Remark \ref{remark:working_with_functor} for a recall of working with functor categories.

\begin{lemma}\label{lemma:facts_on_fpmod}
 \
 \begin{enumerate}
  \item The inclusion $\fp( \PC^{\op}, \Ab ) \subseteq \Modr \PC$ respects cokernels, epis, and monos.
  \item Suppose given a short exact sequence
  \begin{center}
            \begin{tikzpicture}[label/.style={postaction={
              decorate,
              decoration={markings, mark=at position .5 with \node #1;}}},
              mylabel/.style={thick, draw=none, align=center, minimum width=0.5cm, minimum height=0.5cm,fill=white}]
                \coordinate (r) at (2.5,0);
                \coordinate (u) at (0,1.2);
                \node (A) {$0$};
                \node (B) at ($(A)+(r)$) {$F_1$};
                \node (C) at ($(B)+(r)$) {$F_2$};
                \node (D) at ($(C)+(r)$) {$F_3$};
                \node (E) at ($(D)+(r)$) {$0$};
                \draw[->,thick] (A) -- (B);
                \draw[->,thick] (B) -- (C);
                \draw[->,thick] (C) -- (D);
                \draw[->,thick] (D) -- (E);
            \end{tikzpicture}
  \end{center}
  in $\Modr\PC$. If $F_3$ is finitely presented and if $F_2$ is finitely generated, then
  $F_1$ is also finitely generated.
 \end{enumerate}
\end{lemma}
\begin{proof}
 \
 $(1)$:
 Let $F$, $G$ be finitely presented functors with
 presentations $(-,A) \rightarrow (-,A')$ and $(-,B) \rightarrow (-,B')$, respectively.
 A morphism $\nu: F \rightarrow G$ lifts to a morphism $(-,A') \rightarrow (-,B')$,
 since representable functors are projectives in $\Modr\PC$ by Remark \ref{remark:working_with_functor}.
 Computing pointwise, we see that the cokernel of $\nu$ in $\Modr\PC$ is given by the cokernel
 of $(-,A' \oplus B) \rightarrow (-,B')$, and thus, it is finitely presented.
 So, the inclusion respects cokernels and in particular epis.
 Furthermore, we have the following equivalences:
 \begin{align*}
  \nu \text{ is a mono in $\fp(\PC^{\op}, \Ab)$}
  &~\Longleftrightarrow~
  \forall \tau: T \rightarrow F \in \fp(\PC^{\op}, \Ab): (\tau \cdot \nu = 0) \Rightarrow (\tau = 0) \\
  &~\Longleftrightarrow~
  \forall A \in \PC: \forall x \in F(A): \big( (-,A) \stackrel{x}{\rightarrow} F \stackrel{\nu}{\rightarrow} G = 0 \big) \Rightarrow (x = 0) \\
  &~\Longleftrightarrow~
  \forall A \in \PC: \forall x \in F(A): (\nu(x) = 0) \Rightarrow (x = 0) \\
  &~\Longleftrightarrow~
  \nu \text{ is a mono in $\Modr\PC$},
 \end{align*}
 where we identify elements in $F(A)$ with their corresponding natural transformations due to the Yoneda lemma.

 $(2)$: The proof is the same as for modules over a ring, but now in the context of functors.
 Let $(-,A) \rightarrow (-,A')$ be a presentation of $F_3$.
 Then we get a commutative diagram with exact rows
 \begin{center}
            \begin{tikzpicture}[label/.style={postaction={
              decorate,
              decoration={markings, mark=at position .5 with \node #1;}}},
              mylabel/.style={thick, draw=none, align=center, minimum width=0.5cm, minimum height=0.5cm,fill=white}]
                \coordinate (r) at (3,0);
                \coordinate (u) at (0,2);
                \node (A) {};
                \node (B) at ($(A)+(r)$) {$(-,A)$};
                \node (C) at ($(B)+(r)$) {$(-,A')$};
                \node (D) at ($(C)+(r)$) {$F_3$};
                \node (E) at ($(D)+(r)$) {$0$};
                
                \node (A2) at ($(A)-(u)$) {$0$};
                \node (B2) at ($(A2)+(r)$) {$F_1$};
                \node (C2) at ($(B2)+(r)$) {$F_2$};
                \node (D2) at ($(C2)+(r)$) {$F_3$};
                \node (E2) at ($(D2)+(r)$) {$0$};
                
                \draw[->,thick] (B) -- (C);
                \draw[->,thick] (C) -- (D);
                \draw[->,thick] (D) -- (E);
                
                \draw[->,thick] (A2) -- (B2);
                \draw[->,thick] (B2) -- (C2);
                \draw[->,thick] (C2) -- (D2);
                \draw[->,thick] (D2) -- (E2);
                
                \draw[->,thick] (B) --node[left]{$\alpha$} (B2);
                \draw[->,thick] (C) --node[left]{$\beta$} (C2);
                \draw[->,thick] (D) --node[right]{$\id$} (D2);
            \end{tikzpicture}
  \end{center}
  by the projectivity of $(-,A')$ and the universal property of the kernel of $F_2 \rightarrow F_3$.
  The snake lemma implies
  \[
   \cokernel( \beta ) \simeq \cokernel( \alpha ).
  \]
  Since $F_2$ is finitely generated, it admits an epimorphism $(-,B) \twoheadrightarrow F_2$
  and so does $\cokernel( \beta ) \simeq \cokernel( \alpha )$.
  Now, from a projective lift
  \begin{center}
            \begin{tikzpicture}[label/.style={postaction={
              decorate,
              decoration={markings, mark=at position .5 with \node #1;}}},
              mylabel/.style={thick, draw=none, align=center, minimum width=0.5cm, minimum height=0.5cm,fill=white}]
                \coordinate (r) at (3,0);
                \coordinate (u) at (0,2);
                
                \node (B) {};
                \node (C) at ($(B)+(r)$) {};
                \node (D) at ($(C)+(r)$) {$(-,B)$};
                \node (E) at ($(D)+(r)$) {};
                
                \node (B2) at ($(B)-(u)$) {$(-,A)$};
                \node (C2) at ($(B2)+(r)$) {$F_1$};
                \node (D2) at ($(C2)+(r)$) {$\cokernel( \alpha )$};
                \node (E2) at ($(D2)+(r)$) {$0$};

                \draw[->,thick] (B2) --node[above]{$\alpha$} (C2);
                \draw[->,thick] (C2) -- (D2);
                \draw[->,thick] (D2) -- (E2);
                
                \draw[->>,thick] (D) -- (D2);
                \draw[->,thick,dashed] (D) --node[left,above]{$\lambda$} (C2);
            \end{tikzpicture}
  \end{center}
  we can finally construct our desired epimorphism $(-,A\oplus B ) \twoheadrightarrow F_1$.
\end{proof}

\begin{proof}[Proof of the equivalence of statements $(1),(4)-(7)$ in Theorem \ref{theorem:abelian_case}] 
 \
 
 $(1) \implies (4)$:
 If $\PC$ has weak kernels, then Freyd has shown that $\fp( \PC^{\op}, \Ab )$ is abelian (see \cite{PosFreyd} for a constructive proof).
 
 $(4) \implies (5)$:
 Since the inclusion $\fp( \PC^{\op}, \Ab ) \subseteq \Modr \PC$
 respects cokernels and epi-mono factorizations (in particular images) by Lemma \ref{lemma:facts_on_fpmod}, $\fp( \PC^{\op}, \Ab )$
 satisfies the characterization of Corollary \ref{corollary:characterization_as_cokernel_image_closure}.
 
 $(5) \implies (6)$: trivial.
 
 $(6) \implies (7)$: $\CIC( \PC )$ has epi-mono factorizations by Construction \ref{construction:epi_mono_factorization}.
 
 $(7) \implies (1)$:
 Given a morphism $\alpha: A \rightarrow B$ in $\PC$, 
 compute the epi-mono factorization
 \[
  (-,A) \twoheadrightarrow I \hookrightarrow (-,B)
 \]
 of $(-,\alpha)$ in $\fp(\PC^{\op}, \Ab)$. Since the embedding
 $\fp( \PC^{\op}, \Ab ) \subseteq \Modr \PC$ respects epis and monos by Lemma \ref{lemma:facts_on_fpmod},
 $I$ is the image of $(-,\alpha)$ considered as a morphism in $\Modr \PC$ and as such is given by
 \[
  I \simeq (-,A)/\kernel(-,\alpha) \in \fp( \PC^{\op}, \Ab ).
 \]
 In the short exact sequence in $\Modr \PC$
 \begin{center}
            \begin{tikzpicture}[label/.style={postaction={
              decorate,
              decoration={markings, mark=at position .5 with \node #1;}}},
              mylabel/.style={thick, draw=none, align=center, minimum width=0.5cm, minimum height=0.5cm,fill=white}]
                \coordinate (r) at (2.5,0);
                \coordinate (u) at (0,1.2);
                \node (A) {$0$};
                \node (B) at ($(A)+(r)$) {$\ker(-,\alpha)$};
                \node (C) at ($(B)+(r)$) {$(-,A)$};
                \node (D) at ($(C)+(r)$) {$I$};
                \node (E) at ($(D)+(r)$) {$0$};
                \draw[->,thick] (A) -- (B);
                \draw[->,thick] (B) -- (C);
                \draw[->,thick] (C) -- node[above]{$\epsilon_A$} (D);
                \draw[->,thick] (D) -- (E);
            \end{tikzpicture}
  \end{center}
  the object $I$ is finitely presented and $(-,A)$ is finitely generated. Thus, $\ker(-,\alpha)$
  is finitely generated by Lemma \ref{lemma:facts_on_fpmod} and we get an epimorphism
  \[
   (-,K) \twoheadrightarrow \ker(-,\alpha).
  \]
  Now, the composite
  \[
   (-,K) \twoheadrightarrow \ker(-,\alpha) \hookrightarrow (-,A)
  \]
  corresponds via the Yoneda lemma to a morphism
  \[
   K \rightarrow A
  \]
  in $\PC$.
  We claim that this morphism is a weak kernel projection of $\alpha$.
  Given a test morphism $\tau: T \rightarrow A$ in $\PC$ such that $\tau \cdot \alpha = 0$,
  we get a commutative diagram in $\Modr \PC$
  
  \begin{center}
            \begin{tikzpicture}[label/.style={postaction={
              decorate,
              decoration={markings, mark=at position .5 with \node #1;}}},
              mylabel/.style={thick, draw=none, align=center, minimum width=0.5cm, minimum height=0.5cm,fill=white}]
                \coordinate (r) at (2.5,0);
                \coordinate (u) at (0,1.2);
                \node (K) {$(-,K)$};
                \node (ker) at ($(K)+(r)-(u)$) {$\ker(-,\alpha)$};
                \node (A) at ($(ker)+(r)$) {$(-,A)$};
                \node (B) at ($(A)+1.5*(r)$) {$(-,B)$};
                \node (T) at ($(K)-2*(u)$) {$(-,T)$};
                
                \draw[->>,thick] (K) -- (ker);
                \draw[->,thick] (A) --node[above]{$(-,\alpha)$} (B);
                \draw[->,thick] (T) -- (ker);
                \draw[->,thick,dashed] (T) -- (K);
                \draw[->,thick] (ker) -- (A);
                
                \draw (K) [->,thick,out=0,in=155] to (A);
                \draw (T) [->,thick,out=0,in=-155] to (A);
            \end{tikzpicture}
  \end{center}
  since $(-,T)$ is projective. Now, by the Yoneda lemma, the dashed morphism arises from a uniquely determined morphism $T \rightarrow K$ in $\PC$.
\end{proof}

\section{Computational applications}\label{section:applications}

\subsection{A non-coherent ring with decidable syzygy inclusion}
Let $k$ be a
field. In this subsection, we study the ring
\[
 R := k[x_i, z \mid i \in \N]/ \langle x_iz \mid i \in \N \rangle.
\]
from a computational point of view.

\begin{remark}\label{remark:not_coherent}
 $R$ is not a coherent ring, since the kernel of the $R$-module homomorphism
 \[
  R \longrightarrow R: r \mapsto r \cdot \overline{z}
 \]
 is given by
 \[
  \langle \overline{x_i} \mid i \in \N \rangle_R,
 \]
 which cannot be finitely generated as an $R$-module.
\end{remark}

It follows that $\Rows_R$ does not have weak kernels\footnote{
  A weak kernel embedding of the morphism $R^{1 \times 1} \stackrel{\overline{z}}\rightarrow R^{1 \times 1}$ in $\Rows_R$
  would be a column $R^{1 \times m} \rightarrow R^{1 \times 1}$ whose $m \in \Nzero$ entries span the kernel of $R^{1 \times 1} \stackrel{\overline{z}}\rightarrow R^{1 \times 1}$
  in $R\Modl$
  which is impossible by Remark \ref{remark:not_coherent}.
}. 
From Theorem \ref{theorem:abelian_case}, we can conclude
that $\CIC( \Rows_R )$ is not abelian, and we cannot expect to compute kernels in this category.
However, the following theorem implies that we can nevertheless perform all the constructions listed in Section \ref{section:subquotients}
within $\CIC( \Rows_R )$.

\begin{theorem}\label{theorem:rows_R_decidable_syzygy_inclusion}
 If $k$ is a computable field, then
 the category $\Rows_R$ has decidable syzygy inclusion.
\end{theorem}

For the proof, we proceed in three steps.

\begin{enumerate}
 \item We give a simplification of the syzygy inclusion problem for an arbitrary additive category $\PC$ (Corollary \ref{corollary:simplified_syzygy_inclusion_problem}).
 \item We give an explicit description of the row syzygies for matrices over $R$ (Lemma \ref{lemma:decomposition_of_kernel_for_R}).
 \item We solve the simplified syzygy inclusion problem for $\Rows_R$ (Subsubsection \ref{subsubsection:proving_syz_incl_for_R}).
\end{enumerate}

\subsubsection{Simplifying the syzygy inclusion problem}

\begin{lemma}
 Let $\PC$ be an additive category.
 Let
 \begin{center}
    \begin{tikzpicture}[label/.style={postaction={
      decorate,
      decoration={markings, mark=at position .5 with \node #1;}}},
      mylabel/.style={thick, draw=none, align=center, minimum width=0.5cm, minimum height=0.5cm,fill=white}]
        \coordinate (r) at (2.5,0);
        \coordinate (u) at (0,-0.5);
        \node (A) {$A$};
        \node (B) at ($(A)+(r) - (u)$) {$B$};
        \node (C) at ($(B) + (r)$) {$C$};
        \node (Bp) at ($(A) +(r) + (u)$) {$B'$};
        \node (Cp) at ($(Bp) +(r) $) {$C'$};
        \draw[->,thick] (A) --node[above]{$\gamma$} (B);
        \draw[->,thick] (C) --node[above]{$\rho$} (B);
        \draw[->,thick] (A) --node[below]{$\gamma'$} (Bp);
        \draw[->,thick] (Cp) --node[below]{$\rho'$} (Bp);
    \end{tikzpicture}
 \end{center} 
 be a pair of cospans in $\PC$ with the same first object.
 Then
 \[
  \SyzC\big( A \stackrel{\gamma}{\longrightarrow} B \stackrel{\rho}{\longleftarrow} C \big)
  \subseteq
  \SyzC\big( A \stackrel{\gamma'}{\longrightarrow} B' \stackrel{\rho'}{\longleftarrow} C' \big)
 \]
 if and only if
 \[
  \SyzC\big( A \oplus C \stackrel{\pmatcol{\gamma}{\rho}}{\longrightarrow} B \longleftarrow 0 \big)
  \subseteq
  \SyzC\big( A \oplus C \stackrel{\pmatcol{\gamma'}{0}}{\longrightarrow} B' \stackrel{\rho'}{\longleftarrow} C' \big).
 \]
\end{lemma}
\begin{proof}
 ``$\Longrightarrow$'':
 Given a syzygy 
 \begin{center}
          \begin{tikzpicture}[label/.style={postaction={
            decorate,
            decoration={markings, mark=at position .5 with \node #1;}}}, baseline = (D),
            mylabel/.style={thick, draw=none, align=center, minimum width=0.5cm, minimum height=0.5cm,fill=white}]
              \coordinate (r) at (2.5,0);
              \coordinate (u) at (0,1.2);
              \node (A) {$A \oplus C$};
              \node (B) at ($(A)+1*(r)$) {$B$};
              \node (C) at ($(B) +(r)$) {$0$,};
              \node (D) at ($(A) + (u)$) {$S$};
              \draw[->,thick] (A) -- node[below]{$\pmatcol{\gamma}{\rho}$} (B);
              \draw[->,thick] (C) -- node[below]{} (B);
              \draw[->,thick] (D) -- node[left]{$\pmatrow{\sigma_A}{\sigma_C}$} (A);
              \draw[->,thick,dashed] (D) -- node[above,yshift=0.3em]{} (C);
          \end{tikzpicture}
 \end{center}
 we can construct another one:
 \begin{center}
          \begin{tikzpicture}[label/.style={postaction={
            decorate,
            decoration={markings, mark=at position .5 with \node #1;}}}, baseline = (D),
            mylabel/.style={thick, draw=none, align=center, minimum width=0.5cm, minimum height=0.5cm,fill=white}]
              \coordinate (r) at (2.5,0);
              \coordinate (u) at (0,1.2);
              \node (A) {$A$};
              \node (B) at ($(A)+1*(r)$) {$B$};
              \node (C) at ($(B) +(r)$) {$C$.};
              \node (D) at ($(A) + (u)$) {$S$};
              \draw[->,thick] (A) -- node[below]{${\gamma}$} (B);
              \draw[->,thick] (C) -- node[below]{$\rho$} (B);
              \draw[->,thick] (D) -- node[left]{${\sigma_A}$} (A);
              \draw[->,thick,dashed] (D) -- node[above,yshift=0.3em]{$-\sigma_C$} (C);
          \end{tikzpicture}
 \end{center}   
 By assumption, this gives us the syzygy witness $\omega$ in the diagram
 \begin{center}
          \begin{tikzpicture}[label/.style={postaction={
            decorate,
            decoration={markings, mark=at position .5 with \node #1;}}}, baseline = (D),
            mylabel/.style={thick, draw=none, align=center, minimum width=0.5cm, minimum height=0.5cm,fill=white}]
              \coordinate (r) at (2.5,0);
              \coordinate (u) at (0,1.2);
              \node (A) {$A$};
              \node (B) at ($(A)+1*(r)$) {$B'$};
              \node (C) at ($(B) +(r)$) {$C'$,};
              \node (D) at ($(A) + (u)$) {$S$};
              \draw[->,thick] (A) -- node[below]{${\gamma'}$} (B);
              \draw[->,thick] (C) -- node[below]{$\rho'$} (B);
              \draw[->,thick] (D) -- node[left]{${\sigma_A}$} (A);
              \draw[->,thick,dashed] (D) -- node[above,yshift=0.3em]{$\omega$} (C);
          \end{tikzpicture}
 \end{center}
 which finally yields the desired syzygy
 \begin{center}
          \begin{tikzpicture}[label/.style={postaction={
            decorate,
            decoration={markings, mark=at position .5 with \node #1;}}}, baseline = (D),
            mylabel/.style={thick, draw=none, align=center, minimum width=0.5cm, minimum height=0.5cm,fill=white}]
              \coordinate (r) at (2.5,0);
              \coordinate (u) at (0,1.2);
              \node (A) {$A \oplus C$};
              \node (B) at ($(A)+1*(r)$) {$B'$};
              \node (C) at ($(B) +(r)$) {$C'$.};
              \node (D) at ($(A) + (u)$) {$S$};
              \draw[->,thick] (A) -- node[below]{$\pmatcol{\gamma'}{0}$} (B);
              \draw[->,thick] (C) -- node[below]{$\rho'$} (B);
              \draw[->,thick] (D) -- node[left]{$\pmatrow{\sigma_A}{\sigma_C}$} (A);
              \draw[->,thick,dashed] (D) -- node[above,yshift=0.3em]{$\omega$} (C);
          \end{tikzpicture}
 \end{center}
 
 ``$\Longleftarrow$'':
 Given a syzygy 
 \begin{center}
          \begin{tikzpicture}[label/.style={postaction={
            decorate,
            decoration={markings, mark=at position .5 with \node #1;}}}, baseline = (D),
            mylabel/.style={thick, draw=none, align=center, minimum width=0.5cm, minimum height=0.5cm,fill=white}]
              \coordinate (r) at (2.5,0);
              \coordinate (u) at (0,1.2);
              \node (A) {$A$};
              \node (B) at ($(A)+1*(r)$) {$B$};
              \node (C) at ($(B) +(r)$) {$C$};
              \node (D) at ($(A) + (u)$) {$S$};
              \draw[->,thick] (A) -- node[below]{${\gamma}$} (B);
              \draw[->,thick] (C) -- node[below]{$\rho$} (B);
              \draw[->,thick] (D) -- node[left]{${\sigma}$} (A);
              \draw[->,thick,dashed] (D) -- node[above,yshift=0.3em]{$\omega$} (C);
          \end{tikzpicture}
 \end{center}   
 we can construct another one:
 \begin{center}
          \begin{tikzpicture}[label/.style={postaction={
            decorate,
            decoration={markings, mark=at position .5 with \node #1;}}}, baseline = (D),
            mylabel/.style={thick, draw=none, align=center, minimum width=0.5cm, minimum height=0.5cm,fill=white}]
              \coordinate (r) at (2.5,0);
              \coordinate (u) at (0,1.2);
              \node (A) {$A \oplus C$};
              \node (B) at ($(A)+1*(r)$) {$B$};
              \node (C) at ($(B) +(r)$) {$0$.};
              \node (D) at ($(A) + (u)$) {$S$};
              \draw[->,thick] (A) -- node[below]{$\pmatcol{\gamma}{\rho}$} (B);
              \draw[->,thick] (C) -- node[below]{} (B);
              \draw[->,thick] (D) -- node[left]{$\pmatrow{\sigma}{-\omega}$} (A);
              \draw[->,thick,dashed] (D) -- node[above,yshift=0.3em]{} (C);
          \end{tikzpicture}
 \end{center}
 By assumption, we get the syzygy witness $\omega'$ in the diagram
 \begin{center}
          \begin{tikzpicture}[label/.style={postaction={
            decorate,
            decoration={markings, mark=at position .5 with \node #1;}}}, baseline = (D),
            mylabel/.style={thick, draw=none, align=center, minimum width=0.5cm, minimum height=0.5cm,fill=white}]
              \coordinate (r) at (2.5,0);
              \coordinate (u) at (0,1.2);
              \node (A) {$A \oplus C$};
              \node (B) at ($(A)+1*(r)$) {$B'$};
              \node (C) at ($(B) +(r)$) {$C'$};
              \node (D) at ($(A) + (u)$) {$S$};
              \draw[->,thick] (A) -- node[below]{$\pmatcol{\gamma'}{0}$} (B);
              \draw[->,thick] (C) -- node[below]{$\rho'$} (B);
              \draw[->,thick] (D) -- node[left]{$\pmatrow{\sigma}{-\omega}$} (A);
              \draw[->,thick,dashed] (D) -- node[above,yshift=0.3em]{$\omega'$} (C);
          \end{tikzpicture}
 \end{center}
 and obtain the desired syzygy
 
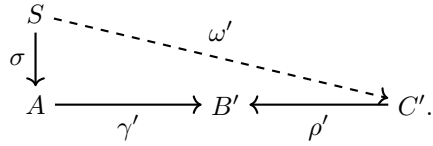
\begin{center}
          \begin{tikzpicture}[label/.style={postaction={
            decorate,
            decoration={markings, mark=at position .5 with \node #1;}}}, baseline = (D),
            mylabel/.style={thick, draw=none, align=center, minimum width=0.5cm, minimum height=0.5cm,fill=white}]
              \coordinate (r) at (2.5,0);
              \coordinate (u) at (0,1.2);
              \node (A) {$A$};
              \node (B) at ($(A)+1*(r)$) {$B'$};
              \node (C) at ($(B) +(r)$) {$C'$.};
              \node (D) at ($(A) + (u)$) {$S$};
              \draw[->,thick] (A) -- node[below]{${\gamma'}$} (B);
              \draw[->,thick] (C) -- node[below]{$\rho'$} (B);
              \draw[->,thick] (D) -- node[left]{${\sigma}$} (A);
              \draw[->,thick,dashed] (D) -- node[above,yshift=0.3em]{$\omega'$} (C);
          \end{tikzpicture}
 \end{center}
\end{proof}

\begin{corollary}[Simplifying the syzygy inclusion problem]\label{corollary:simplified_syzygy_inclusion_problem}
 Let $\PC$ be an additive category.
 Then $\PC$ has decidable syzygy inclusion if and only if
 we can create a solution of the syzygy inclusion problem 
 for all pairs of cospans of the special form
 
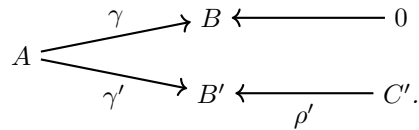
\begin{center}
    \begin{tikzpicture}[label/.style={postaction={
      decorate,
      decoration={markings, mark=at position .5 with \node #1;}}},
      mylabel/.style={thick, draw=none, align=center, minimum width=0.5cm, minimum height=0.5cm,fill=white}]
        \coordinate (r) at (2.5,0);
        \coordinate (u) at (0,-0.5);
        \node (A) {$A$};
        \node (B) at ($(A)+(r) - (u)$) {$B$};
        \node (C) at ($(B) + (r)$) {$0$};
        \node (Bp) at ($(A) +(r) + (u)$) {$B'$};
        \node (Cp) at ($(Bp) +(r) $) {$C'$.};
        \draw[->,thick] (A) --node[above]{$\gamma$} (B);
        \draw[->,thick] (C) -- (B);
        \draw[->,thick] (A) --node[below]{$\gamma'$} (Bp);
        \draw[->,thick] (Cp) --node[below]{$\rho'$} (Bp);
    \end{tikzpicture}
 \end{center} 
\end{corollary}

\subsubsection{Describing row syzygies of matrices over $R$}
We define several computable subrings of $R = k[x_i, z \mid i \in \N]/ \langle x_iz \mid i \in \N \rangle$ that help us in computing row syzygies.
We set
\[
 R_n := k[x_1, \dots, x_n, z]/ \langle x_1z, \dots, x_nz \rangle
\]
for $n \in \N$
which identifies both as a subring and as a quotient ring of $R$.
Moreover, we will regard the polynomial rings
$
k[{x}] := k[x_i \mid i \in \N]
$
and
$
k[z]
$
as subrings of $R$.

\begin{remark}\label{remark:computable_field}
 If $k$ is a computable field, then 
 all the rings $R_n$, $n \in \N$, $k[x]$, and $k[z]$ are also computable.
 For quotients of polynomial rings in finitely many variables like
 $R_n$ and $k[z]$, this follows from Gröbner bases techniques (see, e.g., \cite{GP}).
 For the polynomial ring in infinitely many variables $k[x]$,
 note that $k[x]$ is a free $k[x_1, \dots, x_m]$ module for every $m \in \N$.
 In particular, the inclusion $k[x_1, \dots, x_m] \hookrightarrow k[x]$ is flat,
 which implies that we may compute the row syzygies of a given matrix over $k[x]$
 by computing the row syzygies of the same matrix considered over $k[x_1, \dots, x_m]$
 for sufficiently large $m$.
\end{remark}

\begin{remark}
 We can decompose $R$ at the level of $k$-vector spaces as
 \[
  R = k[x] \oplus \big(z \cdot k[z]\big).
 \]
 For $p \in R$, we write $p = p_x + p_z$ for the corresponding decomposition of the element, i.e.,
 $p_x \in k[x]$ and $p_z \in z \cdot k[z]$.
\end{remark}

For any ring $S$ and any matrix $M \in S^{a \times b}$, $a, b \in \Nzero$,
we write
\[
 \kernel_S( M ) := \{ v \in S^{1 \times a} \mid v \cdot M = 0\}
\]
for the row kernel of $M$.

The next lemma reduces the problem of finding infinitely many generators
for the row syzygies of matrices over $R$ to finding finitely many generators of
row syzygies for matrices over the coherent subrings $R_n$ and $k[x]$.

\begin{lemma}\label{lemma:decomposition_of_kernel_for_R}
 Given a matrix $\left( p^{ij} \right)_{ij} \in R^{a \times b}$ for $a,b \in \Nzero$,
 then we can describe its row kernel as follows:
 \[
  \kernel_R\left( p^{ij} \right)_{ij} =
  \kernel_{R_n}\left( p^{ij} \right)_{ij} \oplus
  \left( \langle x_i \mid i > n \rangle_R \cdot \kernel_{k[x]}\left( p_x^{ij} \right)_{ij} \right)
 \]
 where $n \in \Nzero$ is chosen such that $p^{ij} \in R_n$ for all entries of the given matrix.
\end{lemma}
\begin{proof}
 Given any row $( q^i )_i \in R^{1 \times a}$, we decompose
 its entries w.r.t.\
 \[
  R = R_n \oplus \langle x_i \mid i > n \rangle_R,
 \]
 i.e., 
 \[q^i = s^i + t^i\]
 for $s^i \in R_n$ and $t^i \in \langle x_i \mid i > n \rangle_R$.
 We compute for each $j = 1, \dots, b$
 \begin{align*}
  \sum_{i = 1}^a (s^i + t^i) \cdot (p_x^{ij} + p_z^{ij})
  = ( \sum_{i = 1}^a s^i \cdot p^{ij} ) + ( \sum_{i = 1}^a t^i \cdot p_x^{ij} )
 \end{align*}
 Since $( \sum_{i = 1}^a s^i \cdot p^{ij} ) \in R_n$ and
 $( \sum_{i = 1}^a t^i \cdot p_x^{ij} ) \in \langle x_i \mid i > n \rangle_R$,
 the sum vanishes if and only if both summands vanish.
 It follows that $( q^i )_i \in \kernel_R\left( p^{ij} \right)_{ij}$ if and only if
 \[
  (s^i)_i \in \kernel_{R_n}\left( p^{ij} \right)_{ij}
 \]
 and 
 \[
  (t^i)_i \in \kernel_{R}\left( p_x^{ij} \right)_{ij} \cap (\langle x_i \mid i > n \rangle_R)^{1 \times a}.
 \]
 Finally,
 \begin{align*}
  \kernel_{R}\left( p_x^{ij} \right)_{ij} \cap (\langle x_i \mid i > n \rangle_R)^{1 \times a}
  &=
  \kernel_{k[x]}\left( p_x^{ij} \right)_{ij} \cap (\langle x_i \mid i > n \rangle_{k[x]})^{1 \times a} \\
  &=
  \langle x_i \mid i > n \rangle_R \cdot \kernel_{k[x]}\left( p_x^{ij} \right)_{ij}
 \end{align*}
 due to our choice of $n$.
\end{proof}
%

\subsubsection{Solving the syzygy inclusion problem for $R$}\label{subsubsection:proving_syz_incl_for_R}
We solve the simplified syzygy inclusion problem for $R$,
which, by Corollary \ref{corollary:simplified_syzygy_inclusion_problem},
suffices to solve the syzygy inclusion problem in general, which proves Theorem \ref{theorem:rows_R_decidable_syzygy_inclusion}.
Let 
\[
 R^{1 \times a} \stackrel{\gamma}{\longrightarrow} R^{1 \times b} \longleftarrow 0
\]
and 
\[
 R^{1 \times a} \stackrel{\gamma'}{\longrightarrow} R^{1 \times b'} \stackrel{\gamma'}{\longleftarrow} R^{1 \times c'}
\]
be two cospans in $\Rows_R$ for $a,b,b',c' \in \N$.
Our goal is to decide algorithmically whether
\[
 \SyzC\big( R^{1 \times a} \stackrel{\gamma}{\longrightarrow} R^{1 \times b} \longleftarrow 0 \big) \subseteq \SyzC\big( R^{1 \times a} \stackrel{\gamma'}{\longrightarrow} R^{1 \times b'} \stackrel{\rho'}{\longleftarrow} R^{1 \times c'}\big).
\]
Choose $n \in \N$ such that all entries of $\gamma, \gamma', \rho'$ lie in $R_n$.
Next, compute generators of $\kernel_R( \gamma )$ according to the description in
Lemma \ref{lemma:decomposition_of_kernel_for_R}, i.e.,
compute
finitely many generators
\[
 \sigma_1, \dots, \sigma_d
\]
of $\kernel_{R_n}( \gamma )$ for $d \in \N$,
and finitely many generators
\[
 \tau_1, \dots, \tau_e
\]
of $\kernel_{k[x]}( \gamma_x )$ for $e \in \N$. Note that all $\tau_i$ can be chosen s.t.\ their entries lie in $k[x_1, \dots, x_n]$
by Remark \ref{remark:computable_field}.

\begin{lemma}\label{lemma:checking_only_finitely_many_lifts}
 We have
 \[
 \SyzC\big( R^{1 \times a} \stackrel{\gamma}{\longrightarrow} R^{1 \times b} \longleftarrow 0 \big) \subseteq \SyzC\big( R^{1 \times a} \stackrel{\gamma'}{\longrightarrow} R^{1 \times b'} \stackrel{\rho'}{\longleftarrow} R^{1 \times c'}\big)
\]
 if and only if
 \[
   (R \stackrel{\sigma_i}{\longrightarrow} R^{1 \times a}), (R \stackrel{x_{n+1} \cdot \tau_j}{\longrightarrow} R^{1 \times a}) \in \SyzC\big( R^{1 \times a} \stackrel{\gamma'}{\longrightarrow} R^{1 \times b'} \stackrel{\rho'}{\longleftarrow} R^{1 \times c'}\big)
 \]
 for $i = 1, \dots d$ and $j = 1, \dots, e$.
\end{lemma}
\begin{proof}
 By Lemma \ref{lemma:decomposition_of_kernel_for_R}, the elements $\sigma_i$ and $x_{n+1} \cdot \tau_j$
 lie in $\SyzC\big( R^{1 \times a} \stackrel{\gamma}{\longrightarrow} R^{1 \times b} \longleftarrow 0 \big)$,
 so, we only have to prove the ``$\Longleftarrow$'' direction.
 For this direction, we have to show that an arbitrary syzygy 
 \[
  (R^{1 \times s} {\longrightarrow} R^{1 \times a}) \in
  \SyzC\big( R^{1 \times a} \stackrel{\gamma}{\longrightarrow} R^{1 \times b} \longleftarrow 0 \big)
 \]
 already lies in $\SyzC\big( R^{1 \times a} \stackrel{\gamma'}{\longrightarrow} R^{1 \times b'} \stackrel{\rho'}{\longleftarrow} R^{1 \times c'}\big)$.
 Since such an arbitrary syzygy is nothing but a collection of $s$-many row syzygies,
 we may assume that $s = 1$. So, let
 \[
  (R^{1 \times 1} \stackrel{\sigma}{\longrightarrow} R^{1 \times a}) \in
  \SyzC\big( R^{1 \times a} \stackrel{\gamma}{\longrightarrow} R^{1 \times b} \longleftarrow 0 \big)
 \]
 be a syzygy, which means $\sigma \in \kernel_R( \gamma )$.
 By Lemma \ref{lemma:decomposition_of_kernel_for_R}, we can write $\sigma$ as a sum of the form
 \[
  \sigma = (\sum_{i = 1}^d r_i \cdot \sigma_i) + (\sum_{ \substack{i > n\\ j = 1, \dots, e} } s_{ij} \cdot x_i \cdot \tau_j )
 \]
 for $r_i, s_{ij} \in R$, all but finitely many equal to zero.
 It follows that we only need to prove
 \[
  (R^{1 \times 1} \stackrel{x_i \cdot \tau_j}{\longrightarrow} R^{1 \times a}) \in
  \SyzC\big( R^{1 \times a} \stackrel{\gamma'}{\longrightarrow} R^{1 \times b'} \stackrel{\rho'}{\longleftarrow} R^{1 \times c'}\big)
 \]
 for $i > n + 1 $.
 By assumption, we have 
 \[(R^{1 \times 1} \stackrel{x_{n+1} \cdot \tau_j}{\longrightarrow} R^{1 \times a}) \in \SyzC\big( R^{1 \times a} \stackrel{\gamma'}{\longrightarrow} R^{1 \times b'} \stackrel{\rho'}{\longleftarrow} R^{1 \times c'}\big),\]
 which means that there exists a commutative diagram of the form
 \begin{equation}\label{equation:x_n1_syz}
            \begin{tikzpicture}[label/.style={postaction={
              decorate,
              decoration={markings, mark=at position .5 with \node #1;}}},baseline=(base),
              mylabel/.style={thick, draw=none, align=center, minimum width=0.5cm, minimum height=0.5cm,fill=white}]
                \coordinate (r) at (2.5,0);
                \coordinate (u) at (0,1.2);
                \node (A) {$R^{1 \times 1}$};
                \node (B) at ($(A)+1*(r)$) {};
                \node (C) at ($(B) +(r)$) {};
                \node (A2) at ($(A) - (u)$) {$R^{1 \times a}$};
                \node (B2) at ($(B) - (u)$) {$R^{1 \times b'}$};
                \node (C2) at ($(C) - (u)$) {$R^{1 \times c'}$.};
                \node (base) at ($(A) - 0.5*(u)$) {};
                \draw[->,thick] (A2) -- node[below]{$\gamma'$} (B2);
                \draw[->,thick] (C2) -- node[below]{$\rho'$} (B2);
                \draw[->,thick, dashed] (A) --node[above]{$\omega$}  (C2);
                \draw[->,thick] (A) --node[left]{$x_{n+1} \cdot \tau_j$} (A2);
            \end{tikzpicture}
 \end{equation}
 For any $i > n+1$, we can define a ring automorphism $\phi_i$ of $R$
 by
 \begin{align*}
  \phi_i( z ) &:= z,  \\
  \phi_i( x_{n+1} ) &:= x_i, \\
  \phi_i( x_i ) &:= x_{n+1}, \\
  \phi_i( x_j ) &:= x_{j}, \hspace{1em} j \not\in \{i, n+1\}.\\
 \end{align*}
 Applying $\phi_i$ to the diagram \eqref{equation:x_n1_syz} yields
 \begin{center}
            \begin{tikzpicture}[label/.style={postaction={
              decorate,
              decoration={markings, mark=at position .5 with \node #1;}}},
              mylabel/.style={thick, draw=none, align=center, minimum width=0.5cm, minimum height=0.5cm,fill=white}]
                \coordinate (r) at (2.5,0);
                \coordinate (u) at (0,1.2);
                \node (A) {$R^{1 \times 1}$};
                \node (B) at ($(A)+1*(r)$) {};
                \node (C) at ($(B) +(r)$) {};
                \node (A2) at ($(A) - (u)$) {$R^{1 \times a}$};
                \node (B2) at ($(B) - (u)$) {$R^{1 \times b'}$};
                \node (C2) at ($(C) - (u)$) {$R^{1 \times c'}$};
                \draw[->,thick] (A2) -- node[below]{$\gamma'$} (B2);
                \draw[->,thick] (C2) -- node[below]{$\rho'$} (B2);
                \draw[->,thick, dashed] (A) --node[above]{$\phi_i(\omega)$}  (C2);
                \draw[->,thick] (A) --node[left]{$x_{i} \cdot \tau_j$} (A2);
            \end{tikzpicture}
        \end{center}
  since $\phi_i$ leaves $\gamma', \rho', \tau_j$ invariant due to our choice of $n$.
  This proves that
  $x_i \cdot \tau_j \in \SyzC\big( R^{1 \times a} \stackrel{\gamma'}{\longrightarrow} R^{1 \times b'} \stackrel{\rho'}{\longleftarrow} R^{1 \times c'}\big)$
  for $i > n + 1$ and consequently
  that $\sigma$ is a syzygy in
  $\SyzC\big( R^{1 \times a} \stackrel{\gamma'}{\longrightarrow} R^{1 \times b'} \stackrel{\rho'}{\longleftarrow} R^{1 \times c'}\big)$.
\end{proof}

\begin{proof}[Proof of Theorem \ref{theorem:rows_R_decidable_syzygy_inclusion}]
 Since $k$ is computable, the rings $R_n$ and $k[x]$ are also computable by Remark \ref{remark:computable_field}.
 In particular, we may compute the finitely many elements $\sigma_1, \dots, \sigma_d$
 and $x_{n+1}\cdot\tau_1, \dots, x_{n+1}\cdot\tau_e$
 of Lemma \ref{lemma:checking_only_finitely_many_lifts}.
 By Remark \ref{remark:rows_R_has_decidable_lifts} below,
 $\Rows_R$ has decidable lifts, which means that we can check if these finitely many
 elements are syzygies or not.
\end{proof}

\begin{remark}\label{remark:rows_R_has_decidable_lifts}
 If we start with a diagram
 \begin{center}
            \begin{tikzpicture}[label/.style={postaction={
              decorate,
              decoration={markings, mark=at position .5 with \node #1;}}},
              mylabel/.style={thick, draw=none, align=center, minimum width=0.5cm, minimum height=0.5cm,fill=white}]
                \coordinate (r) at (2.5,0);
                \coordinate (u) at (0,1.2);
                \node (A) {$R^{1 \times a}$};
                \node (B) at ($(A)+1*(r)$) {};
                \node (C) at ($(B) +(r)$) {};
                \node (A2) at ($(A) - (u)$) {$R^{1 \times c}$};
                \node (B2) at ($(B) - (u)$) {$R^{1 \times b}$};
                \draw[->,thick] (B2) -- node[below]{$\beta$} (A2);
                \draw[->,thick] (A) --node[left]{$\alpha$} (A2);
            \end{tikzpicture}
        \end{center}
 in $\Rows_R$, there exists an $n \in \N$ such that all entries of $\alpha$ and $\beta$ lie in
 $R_n$. There exists a lift
 \begin{center}
            \begin{tikzpicture}[label/.style={postaction={
              decorate,
              decoration={markings, mark=at position .5 with \node #1;}}},
              mylabel/.style={thick, draw=none, align=center, minimum width=0.5cm, minimum height=0.5cm,fill=white}]
                \coordinate (r) at (2.5,0);
                \coordinate (u) at (0,1.2);
                \node (A) {$R^{1 \times a}$};
                \node (B) at ($(A)+1*(r)$) {};
                \node (C) at ($(B) +(r)$) {};
                \node (A2) at ($(A) - (u)$) {$R^{1 \times c}$};
                \node (B2) at ($(B) - (u)$) {$R^{1 \times b}$};
                \draw[->,thick] (B2) -- node[below]{$\beta$} (A2);
                \draw[->,thick] (A) --node[left]{$\alpha$} (A2);
                \draw[->,dashed,thick] (A) -- (B2);
            \end{tikzpicture}
        \end{center}
  in $\Rows_R$ if and only if there exists a lift in $\Rows_{R_n}$,
  since we can always apply the natural epimorphism $R \twoheadrightarrow R_n$
  to the entries of a lift in $\Rows_R$ in order to obtain a lift in $\Rows_{R_n}$.
\end{remark}

\subsection{Subcategories of graded modules and functors}

We have seen that the category constructor $\CIC( - )$ applied to $\Rows_R$ for a ring $R$
yields a computational model for a certain subcategory of $R\Modl$.
As a benefit of the abstraction that we made in this paper, 
we give two more examples of additive categories
that yield interesting results when we apply $\CIC(-)$ to them.

\begin{example}[Graded modules]
 Let $G$ be a group and let $S$
 be a $G$-graded ring, i.e., it comes equipped with
 a decomposition into abelian groups $S = \bigoplus_{g \in G}S_g$ such that
 $S_g \cdot S_h \subseteq S_{gh}$ for all $g,h \in G$, and the multiplicative unit of $S$ lies in $S_e$ for $e$ the neutral element of $G$.
 For such a $G$-graded ring, we may define the category $\grRows_S$
 of \textbf{graded left row modules}.
 Its objects are given by direct sums of shifts of $S$ (considered as a graded $S$-module),
 where the shift by $g \in G$ of a graded left $S$-module $M = \bigoplus_{h \in G} M_h$ is defined by
 \[
  M(g) := \bigoplus_{h \in G}M_{h\cdot g}.
 \]
 Morphisms in $\grRows_S$ are given by $G$-graded $S$-module homomorphisms,
 which can be identified with matrices over $S$ having homogeneous entries
 whose degrees are compatible with the shifts occurring in the source and range.
 Concretely, a morphism
 \[S(d_1) \oplus \dots \oplus S(d_r) \rightarrow S(e_1) \oplus \dots \oplus S(e_s)\]
 for $d_1, \dots, d_r, e_1, \dots, e_s \in G$, $r, s \in \Nzero$
 is given by a matrix $(H_{ij})_{ij} \in S^{r \times s}$
 such that $H_{ij} \in S_{d_i^{-1} \cdot e_j}$ for all $i,j$.
 A functor
 \[
  F: \grRows_S^{\op} \longrightarrow \Ab
 \]
 gives rise to a graded left $S$-module $\bigoplus_{g \in G} F(S(g^{-1}))$, and similar to
 the non-graded case described in Example \ref{example:functors_are_modules},
 we have an equivalence between $\Modr \grRows_S$ and the category of graded $S$-modules.
 It follows by Corollary \ref{corollary:characterization_as_cokernel_image_closure}
 that $\CIC( \grRows_S )$ can be seen as a computational model for 
 the smallest full and replete additive subcategory of all graded $S$-modules
 that includes shifts of $S$, cokernels, and images.
\end{example}

\begin{example}[Functors]
 For an additive category $\PC$, the category $\CIC( \PC )$ always has cokernels by Construction \ref{construction:cokernel}.
 Thus, $\CIC( \PC )^{\op}$ has kernels, so in particular weak kernels, which implies
 \[
  \CIC( \CIC( \PC )^{\op} ) \simeq \fp( \CIC( \PC ), \Ab )
 \]
 by Theorem \ref{theorem:abelian_case}.
 Thus, an iterated application of $\CIC( - )$ can yield a computational model for categories of finitely presented functors
 on $\CIC( \PC )$.
\end{example}

\def\cprime{$'$} \def\cprime{$'$} \def\cprime{$'$} \def\cprime{$'$}
  \def\cprime{$'$}

\end{document}